\documentclass{article}

\usepackage[english]{babel}
\usepackage[dvipsnames,usenames]{color}

\usepackage[demo]{graphicx}
\usepackage{tikz-qtree}

\usepackage{stackengine}
\usepackage{scalerel}

\stackMath

\newcommand{\nats}                  {{\mathbb N}}

\newcommand{\calV}{\mathcal{V}}

\newcommand{\BMt} {\mathbb{B}}
\newcommand{\CMt} {\mathbb{C}}
\newcommand{\DMt} {\mathbb{D}}
\newcommand{\Mt} {\mathbb{M}}
\newcommand{\MPt} {\mathbb{MP}}

\usepackage{verbatim,graphics,indentfirst} 
\usepackage{url}
\usepackage{stmaryrd}
\usepackage[nosumlimits]{amsmath}
\usepackage{amssymb}
\usepackage{amsthm}
\usepackage{latexsym}
\usepackage{enumerate}
\usepackage{graphicx}
\usepackage[T1]{fontenc}
\usepackage{hyperref}
\usepackage{mathrsfs} 
\usepackage{setspace}
\usepackage[all]{xy}

 \usepackage{bbm}

\newcommand{\conn}{{\copyright}}

\theoremstyle{definition}
\newtheorem{definition}{\vspace{1mm}Definition}[section]

\theoremstyle{plain}
\newtheorem{lemma}[definition]{\vspace{1mm}Lemma}
\newtheorem{theorem}[definition]{\vspace{1mm}Theorem}

\newtheorem{proposition}[definition]{\vspace{1mm}Proposition}

\newtheorem{example}[definition]{\vspace{1mm}Example}

\newtheorem*{bla1}{Example~\ref{pnegpq}, revisited}
\newtheorem*{bla2}{Example~\ref{pnegpnegq}, revisited}
\newtheorem*{bla3}{Example~\ref{clun}, revisited}
\newtheorem*{bla4}{Example~\ref{agata}, revisited}
\newtheorem*{bla5}{Example~\ref{nelson}, revisited}
\newtheorem*{bla6}{Example~\ref{swap}, revisited}
\newtheorem*{bla7}{Example~\ref{L53}, revisited}

\newcommand{\der}    						  {\vartriangleright}

\newcommand{\tuple}[1]                         {{\langle #1\rangle}}

\DeclareMathOperator*{\sub}{\mathsf{sub}}
\DeclareMathOperator*{\var}{\mathsf{var}}
\DeclareMathOperator*{\prfx}{\mathsf{prfx}}

\DeclareMathOperator*{\inst}{\mathsf{inst}}

\newcommand{\ou}          {\vee}
\newcommand{\e}          {\wedge}

\newcommand{\Val}{\textrm{Val}}

\newcommand{\ignore}[1]{}

\newcommand{\Ax}             {\mathsf{Ax}}
\newcommand{\Int}             {\mathsf{Int}}
\newcommand{\Exp}             {\mathsf{Exp}}
\newcommand{\Norm}             {\mathsf{Norm}}
\newcommand{\Tm}             {\mathsf{Tm}}
 \date{ }

\title{On axioms and rexpansions
\footnote{
This research was funded by FCT/MCTES through national funds and when applicable co-funded EU funds under the project UIDB/EEA/50008/2020.
Work done under the scope of the CaCTus initiative of SQIG at Instituto de Telecomunica\c{c}\~oes.
}
}

\author{Carlos Caleiro and S\'ergio Marcelino\\
{\tt \{ccal,smarcel\}@math.tecnico.ulisboa.pt} \\
{SQIG - Instituto de Telecomunica\c c\~oes}\\
{Dep. Matem\'atica - Instituto Superior T\'ecnico}\\
{Universidade de Lisboa, Portugal}}

\begin{document}

\maketitle

\begin{abstract}
We study the general problem of strengthening the logic of a given (partial) (non-deterministic) matrix with a set of axioms, using the idea of rexpansion. 
We obtain two characterization methods: a very general but not very effective one, and then an effective method which only applies under certain restrictions on the given semantics and the shape of the axioms. We show that this second method covers a myriad of examples in the literature. 
Finally, we illustrate how to obtain analytic multiple-conclusion calculi for the resulting logics.
\end{abstract}

\section{Introduction}
The work reported in this paper has three underlying aims.

First, and foremost, on a higher-level reading, this paper is an acclamation of the modularization power enabled by \emph{non-deterministic matrices (Nmatrices)}, as proposed and developed by Arnon Avron, along with his coauthors and students over the past 15 years~\cite{avlev05,avron2005,avronnegations,avron2007,AvronBK07,Avron,Avron2012atLICS,avroncutfree2013,rexpansions}, and used by many others~\cite{wollic,wollic19,taming,Baaz2013,soco,synt,coniglioswap}  when seeking for a clear semantic rendering of logics resulting from strengthening a given base logic. 

Secondly, in the technical developments we propose, this paper can be seen as an application of the ideas behind \emph{rexpansions}~\cite{rexpansions} of Nmatrices, in the form of a generalization of the systematic method put forth in~\cite{taming} for obtaining modularly a suitable semantics for a given logic strengthened with additional axioms (and new unary connectives). Expectedly, the method may yield in general a \emph{partial non-deterministic matrix (PNmatrix)}~\cite{Baaz2013}, partiality being a feature that adds to the conciseness of Nmatrices but which is known to contend with \emph{analyticity}.

Last but not least, this paper is an opportunity for putting into practice the techniques developed in~\cite{synt,wollic19} for obtaining an analytic multiple-conclusion calculus for the logic defined by any finite PNmatrix (under a reasonable expressiveness proviso).
This is in contrast with comparable results for sequent-like calculi~\cite{Baaz2013,taming}, for which partiality seems to devoid them of a usable (even if generalized) subformula property capable of guaranteeing analyticity (and elimination of non-analytic cuts).\smallskip

The paper is organized as follows. In Section~\ref{sec2}, we recall (or suitably adapt) the necessary notions about logics, their syntax and semantics. Section~\ref{sec3} presents two methods for using rexpansions in order to obtain semantic characterizations of the strenghtening with additional (schema) axioms $\Ax$ of the logic of a given PNmatrix $\Mt$. The first method, presented in Subsection~\ref{sec31}, is completely general but unfortunately produces an infinite PNmatrix even when a finite one would be available. In order to overcome this drawback, in Subsection~\ref{sec32}, we present another more economic method, generalizing~\cite{taming}, which, under suitable requirements, always provides a finite PNmatrix when starting from finite $\Mt$ and $\Ax$. Section~\ref{sec4} is devoted to illustrating the application of the method of Subsection~\ref{sec32} to some meaningful examples.
Then, in Section~\ref{sec5}, we show that (under minimal expressiveness requirements on $\Mt$) the results of~\cite{synt,wollic19} can be used to provide analytic multiple-conclusion calculi to the strengthened logics by exploring the semantics obtained by our method, and provide illustrative examples. We close the paper in Section~\ref{sec6}, with some concluding remarks and topics for future work.

\section{Preliminaries}\label{sec2}

For the sake of self-containment, and in order to fix notation and terminology, we start by recalling (or suitably adapting, or generalizing) a number of useful notions and results. 
Instead of going through this material sequentially, the reader could as well jump this section for the moment and refer back here whenever necessary.\smallskip

A propositional \emph{signature} $\Sigma$ is a family $\{\Sigma^{(k)}\}_{k\in \nats}$ of sets, where each $\Sigma^{(k)}$ contains the \emph{$k$-place connectives} of $\Sigma$. 
To simplify notation, 
we express the fact that $\conn\in\Sigma^{(k)}$ for some $k\in\nats$ by simply writing $\conn\in\Sigma$, and we write $\Sigma'\cup\Sigma$ or 
$\Sigma'\subseteq\Sigma$ to denote the union or the inclusion, respectively, if $\Sigma'$ is also a signature. Given a signature $\Sigma$, the language $L_{\Sigma}(P)$ is the carrier of the absolutely free $\Sigma$-algebra generated over a given 
denumerable 
set of sentential variables $P$. Elements of $L_{\Sigma}(P)$ are called \emph{formulas}.  Given a formula $A\in L_{\Sigma}(P)$, we denote by $\var(A)$ (resp.\ $\sub(A)$) the set of variables (resp.\ subformulas) of~$A$, defined as usual; the extension of $\var$ and $\sub$, and other similar functions, from formulas to sets thereof is defined as expected. A \emph{substitution} is a member $\sigma\in L_\Sigma(P)^P$, that is, a function $\sigma:P\to  L_\Sigma(P)$, 
uniquely extendable into an endomorphism $\cdot^\sigma:L_\Sigma(P)\to  L_\Sigma(P)$.
Given $\Gamma\subseteq L_{\Sigma}(P)$, we denote by~$\Gamma^\sigma$ the set $\{A^\sigma:A\in\Gamma\}$. 
For $A\in L_\Sigma(P)$, define $A^{\inst}=\{A^\sigma:\sigma\in L_\Sigma(P)^P\}$ and $\Gamma^{\inst}=\bigcup\limits_{A\in\Gamma} A^{\inst}$.

Given formulas $A,A_1,\dots,A_n\in L_\Sigma(P)$ with $\var(A)\subseteq\{p_1,\dots,p_n\}$, we write $A(A_1,\dots,A_n)$ to denote the formula $A^\sigma$ where $\sigma(p_i)=A_i$ for $1\leq i\leq n$.\smallskip

Given a signature $\Sigma$, a \emph{$\Sigma$-PNmatrix (partial non-deterministic matrix)} is a structure $\Mt=\tuple{V,D,\cdot_\Mt}$ such that $V$ is a set (of \emph{truth-values}), $D\subseteq V$ is the set of \emph{designated} values, and $\conn_\Mt:V^k\to \wp(V)$ is a function (\emph{truth-table}) for each $k\in\nats$ and each $k$-place connective $\conn\in\Sigma$. When $\conn_\Mt(x_1,\dots,x_k)\neq\emptyset$ for all $x_1,\dots,x_k\in V$ we say that the truth-table of $\conn$ in 
$\Mt$ is \emph{total}. When $\conn_\Mt(x_1,\dots,x_k)$ has at most one element for all $x_1,\dots,x_k\in V$ we say that the truth-table of $\conn$ in 
$\Mt$ is \emph{deterministic}. Of course, deterministic does not imply total. Given $\Sigma'\subseteq\Sigma$, we say that $\Mt$ is \emph{$\Sigma'$-total} if the truth-tables in $\Mt$ of the connectives $\conn\in\Sigma'$ are all total. Analogously, we say that $\Mt$ is \emph{$\Sigma'$-deterministic} if the truth-tables in $\Mt$ of the connectives $\conn\in\Sigma'$ are all deterministic. When the $\Sigma$-PNmatrix $\Mt$ is $\Sigma$-total, or just total, it is simply 
called a \emph{$\Sigma$-Nmatrix, or Nmatrix (non-deterministic matrix)}. When a $\Sigma$-Nmatrix $\Mt$ is $\Sigma$-deterministic, or just deterministic, it is simply 
called a \emph{$\Sigma$-matrix}, or a logical matrix. For the sake of completing the picture, when a $\Sigma$-PNmatrix $\Mt$ is deterministic we call it 
a \emph{$\Sigma$-Pmatrix}, or Pmatrix.

Granted a $\Sigma$-PNmatrix $\Mt=\tuple{V,D,\cdot_\Mt}$, a \emph{$\Mt$-valuation} is a function $v:L_\Sigma(P)\to V$ such that $v(\conn(A_1,\dots,A_k))\in\conn_\Mt(v(A_1),\dots,v(A_k))$ for every $k\in\nats$, every $k$-place connective $\conn\in\Sigma$, and every $A_1,\dots,A_k\in L_\Sigma(P)$. We denote the set of all $\Mt$-valuations by $\Val_\Mt$. Given a formula $A\in L_\Sigma(\{p_1,\dots,p_n\})$, we extend the usual notation for connectives and use $A_\Mt:V^n\to\wp(V)$ to denote the function defined by $A_\Mt(x_1,\dots,x_n)=\{v(A):v\in\Val_\Mt\textrm{ with }v(p_i)=x_i\textrm{ for }1\leq i\leq n\}$ for every $x_1,\dots,x_n\in V$.


As is well known, if $\Mt=\tuple{V,D,\cdot_\Mt}$ is a matrix then every function $f:Q\to V$ with $Q\subseteq P$ can be extended to a $\Mt$-valuation (in an essentially unique way for all formulas $A$ with $\var(A)\subseteq Q$). As a consequence, $A_\Mt(x_1,\dots,x_n)$ is a singleton when $\Mt$ is a matrix, or more generally when there is $\Sigma'\subseteq\Sigma$ such that $A\in L_{\Sigma'}(P)$ and $\Mt$ is $\Sigma'$-deterministic and $\Sigma'$-total. If $\Mt$ is only known to be $\Sigma'$-deterministic, we can at least guarantee that $A_\Mt(x_1,\dots,x_n)$ has at most one element.
When $\Mt$ is a Nmatrix, however, $A_\Mt(x_1,\dots,x_n)$ can be a large (non-empty) set. Still, we know from~\cite{Avron} that a function 
$f:\Gamma\to V$ with $\Gamma\subseteq L_\Sigma(P)$ can be extended to a $\Mt$-valuation provided that $\sub(\Gamma)\subseteq\Gamma$ and 
that $f(\conn(A_1,\dots,A_k))\in\conn_\Mt(f(A_1),\dots,f(A_n))$ whenever $\conn(A_1,\dots,A_k)\in\Gamma$. 
In case $\Mt$ is a PNmatrix, in general, one does not even have such a guarantee~\cite{Baaz2013}, unless $f(\Gamma)\in\mathcal{T}_{\Mt}=\bigcup_{v\in \Val_\Mt}\wp(v(L_\Sigma(P)))$. In other words, given $X\subseteq V$, we have $X\in\mathcal{T}_\Mt$ if the values in $X$ are all together compatible in some valuation of $\Mt$. Of course, $A_\Mt(x_1,\dots,x_n)\neq\emptyset$ if $\{x_1,\dots,x_n\}\in\mathcal{T}_\Mt$.\smallskip

%
%
%
%
%
%

A set of valuations $\calV\subseteq\Val_\Mt$ characterizes a generalized (multiple conclusion) consequence relation $\der_\calV{\subseteq}\wp(L_\Sigma(P))\times\wp(L_\Sigma(P))$ defined by $\Gamma\der_\calV\Delta$ when for every $v\in \calV$ if $v(\Gamma)\subseteq D$ then $v(\Delta)\cap D\neq\emptyset$. Of course, it also defines the more usual (single conclusion) consequence relation
$\vdash_\calV{\subseteq}\wp(L_\Sigma(P))\times L_\Sigma(P)$ such that $\Gamma\vdash_\calV A$ when $\Gamma\der_\calV \{A\}$. In both cases, $\der_\calV$ and $\vdash_\calV$ are \emph{substitution invariant}, and respectively a Scott~\cite{scott} and Shoesmith and Smiley~\cite{ShoesmithSmiley} consequence relation, or else a Tarskian consequence relation, when $\calV$ is closed for substitutions, that is, if $v\in \calV$ and $\sigma\in L_\Sigma(P)^P$ then $v\circ (\,\cdot^\sigma)\in \calV$.

We simply write $\der_\Mt$ or $\vdash_\Mt$, instead of $\der_{\Val_\Mt}$ or $\vdash_{\Val_\Mt}$, respectively, and say that the consequences are characterized by $\Mt$. With respect to given consequence relations $\der$ or $\vdash$, we say that $\Mt$ is \emph{sound} if $\der{\subseteq}\der_\Mt$ or $\vdash{\subseteq}\vdash_\Mt$, and we say that $\Mt$ is \emph{complete} if $\der_\Mt{\subseteq}\der$ or $\vdash_\Mt{\subseteq}\vdash$.\smallskip

A \emph{refinement} of a $\Sigma$-PNmatrix $\Mt=\tuple{V,D,\cdot_\Mt}$ is any $\Sigma$-PNmatrix $\Mt'=\tuple{V',D',\cdot_{\Mt'}}$ with $V'\subseteq V$, $D'=D{\cap} V'$, and $\conn_{\Mt'}(x_1,\dots,x_k)\subseteq\conn_{\Mt}(x_1,\dots,x_k)$ for every $k\in\nats$, every $k$-place connective $\conn\in\Sigma$, and every $x_1,\dots,x_k\in V'$. 
It is clear, almost by definition, that $\Val_{\Mt'}\subseteq\Val_\Mt$. 
When it is always the case that $\conn_{\Mt'}(x_1,\dots,x_k)=\conn_{\Mt}(x_1,\dots,x_k)\cap V'$ then the refinement is called \emph{simple} and $\Mt'$ is denoted by $\Mt_{V'}$. 
Clearly, $v\in\Val_\Mt$ implies that $v\in\Val_{\Mt_{V'}}$ with $V'=v(L_\Sigma(P))$, and also that $\Mt_{V'}$ is a non-empty total refinement of $\Mt$. This observation justifies the equivalent definition of $\mathcal{T}_\Mt$ put forth in~\cite{wollic19}.

$\mathcal{E}:V\to\wp(U)$ is an \emph{expansion function} if $\mathcal{E}(x)\neq\emptyset$ for every $x\in V$, and $\mathcal{E}(x)\cap\mathcal{E}(x')=\emptyset$ if $x'\in V$ is distinct from $x$. Given $X\subseteq V$, we abuse notation and use $\mathcal{E}(X)$ to denote $\bigcup_{x\in X}\mathcal{E}(x)$. One associates to $\mathcal{E}$ its \emph{contraction} 
$\widetilde{\mathcal{E}}:\mathcal{E}(V)\to V$ such that, for each $y\in \mathcal{E}(V)$, $\widetilde{\mathcal{E}}(y)\in V$ is the unique such that $y\in\mathcal{E}(\widetilde{\mathcal{E}}(y))$. The \emph{$\mathcal{E}$-expansion} of a $\Sigma$-PNmatrix $\Mt=\tuple{V,D,\cdot_\Mt}$ is the $\Sigma$-PNmatrix $\mathcal{E}(\Mt)=\tuple{\mathcal{E}(V),\mathcal{E}(D),\cdot_{\mathcal{E}(\Mt)}}$ such that 
$\conn_{\mathcal{E}(\Mt)}(y_1,\dots,y_k)=\mathcal{E}(\conn_{\Mt}(\widetilde{\mathcal{E}}(y_1),\dots,\widetilde{\mathcal{E}}(y_k)))$ for every $k\in\nats$, every $k$-place connective $\conn\in\Sigma$, and every $y_1,\dots,y_k\in \mathcal{E}(V)$. By construction, it is clear that $\widetilde{\mathcal{E}}$ preserves and reflects designated values, i.e., $\widetilde{\mathcal{E}}(y)\in D$ if and only if $y\in \mathcal{E}(D)$. Further, given a function $f:L_\Sigma(P)\to\mathcal{E}(V)$, $f\in\Val_{\mathcal{E}(\Mt)}$ if and only if $\widetilde{\mathcal{E}}\circ f\in\Val_\Mt$.

A \emph{rexpansion} of a $\Sigma$-PNmatrix $\Mt=\tuple{V,D,\cdot_\Mt}$ is a refinement of some $\mathcal{E}$-expansion of $\Mt$. When $\Mt^\dag=\tuple{V^\dag,D^\dag,\cdot_{\Mt^\dag}}$ is a rexpansion of $\Mt$, we still have that if $v^\dag\in\Val_{\Mt^\dag}$ then $\widetilde{\mathcal{E}}\circ v^\dag\in\Val_{\Mt}$. {Consequently, we have that 
$\widetilde{\mathcal{E}}(A_{\Mt^\dag}(x_1,\dots,x_n))\subseteq A_{\Mt}(\widetilde{\mathcal{E}}(x_1),\dots,\widetilde{\mathcal{E}}(x_n))$, for every  $A\in L_\Sigma(\{p_1,\dots,p_n\})$ and $x_1,\dots,x_n\in V^\dag$.} 

It is easy to see that the refinement relation, the expansion relation, and thus also the rexpansion relation, are all transitive.\smallskip

We end this section with a very simple but useful lemma.

\begin{lemma}\label{easylemma}
Let $\Sigma'\subseteq\Sigma$ and $\Mt=\tuple{V,D,\cdot_\Mt}$ be a $\Sigma'$-deterministic $\Sigma$-PNmatrix.

If $\Mt^\dag=\tuple{V^\dag,D^\dag,\cdot_{\Mt^\dag}}$ is a rexpansion of $\Mt$, $A\in L_{\Sigma'}(\{p_1,\dots,p_n\})$, and $y,z\in A_{\Mt^\dag}(x_1,\dots,x_n)$ then $y\in D^\dag$ if and only if $z\in D^\dag$.
\end{lemma}
\begin{proof}
Assume that $\Mt^\dag$ is a refinement of the expansion of $\Mt$ with $\mathcal{E}$. If $y,z\in A_{\Mt^\dag}(x_1,\dots,x_n)$ then 
$\widetilde{\mathcal{E}}(y),\widetilde{\mathcal{E}}(z)\in \widetilde{\mathcal{E}}(A_{\Mt^\dag}(x_1,\dots,x_n))\subseteq A_\Mt(\widetilde{\mathcal{E}}(x_1),\dots,\widetilde{\mathcal{E}}(x_n))$. Since $\Mt$ is $\Sigma'$-deterministic and $A\in L_{\Sigma'}(P)$ it follows that  $A_\Mt(\widetilde{\mathcal{E}}(x_1),\dots,\widetilde{\mathcal{E}}(x_n))$ has at most one element, and thus $\widetilde{\mathcal{E}}(y)=\widetilde{\mathcal{E}}(z)$. Therefore, $y\in D^\dag$ iff $\widetilde{\mathcal{E}}(y)\in D$ iff $\widetilde{\mathcal{E}}(z)\in D$ iff $z\in D^\dag$.
\end{proof}


\section{Adding axioms}\label{sec3}

Given a signature $\Sigma$, a Tarskian consequence relation $\vdash$ over $\Sigma$, and $\Ax\subseteq L_\Sigma(P)$, the \emph{strengthening of ${\vdash}$ with (schema) axioms $\Ax$} is the consequence relation 
$\vdash^{\Ax}$ defined by $\Gamma\vdash^{\Ax}A$ if and only if $\Gamma\cup\Ax^{\inst}\vdash A$.\\

Our aim is to provide an adequate (and usable) semantics for $\vdash^{\Ax}$, given a semantic characterization of $\vdash$, a task that is well within the general effort of characterizing combined logics~\cite{ccal:wcarnielli:jfr:css:04,deccomp,wollic}. The following simple result, whose (simple) proof we omit, is a corollary of Lemma~2.7 of~\cite{soco}.

\begin{proposition}\label{proofbyvals}
Let $\Mt=\tuple{V,D,\cdot_\Mt}$ be a $\Sigma$-PNmatrix and $\Ax\subseteq L_\Sigma(P)$.
The consequence relation $\vdash^{\Ax}_\Mt$ is characterized by $\Val_\Mt^\Ax=\{v\in\Val_\Mt:v(\Ax^{\inst})\subseteq D\}$.
\end{proposition}

Our aim in the forthcoming subsections is to design some systematic way of using the ideas behind rexpansions for transforming $\Mt$ into a PNmatrix whose valuations somehow coincide with $\Val_\Mt^\Ax$.

\subsection{A general construction}\label{sec31}

As a first attempt, we employ a general technique from the theory of combining logics~\cite{ccal:wcarnielli:jfr:css:04,deccomp,wollic}. The overall idea, when starting from a given PNmatrix and a set of strengthening axioms, is to pair each formula of the logic with its possible values but guaranteeing that instances of axioms can only be paired with designated values. 

\begin{theorem}\label{flat}
Let $\Mt=\tuple{V,D,\cdot_\Mt}$ be a $\Sigma$-PNmatrix and $\Ax\subseteq L_\Sigma(P)$.\\
The consequence $\vdash_\Mt^\Ax$ is characterized by the rexpansion $\Mt^\flat_\Ax=\tuple{V^\flat_\Ax,D^\flat_\Ax,\cdot_{\Mt^\flat_\Ax}}$ of $\Mt$ defined by:	
\begin{itemize}
\item $V^\flat_\Ax=\{(x,A)\in V\times L_\Sigma(P):\textrm{ if }A\in\Ax^{\inst}\textrm{ then }x\in D\}$,
\item $D^\flat_\Ax=D\times L_{\Sigma}(P)$,
\item for each $k\in\nats$ and $\conn\in\Sigma^{(k)}$, $$\conn_{\Mt^\flat_\Ax}((x_1,A_1),\ldots,(x_k,A_k))=$$
$$\{(x,\conn(A_1,\dots,A_k))\in V^\flat_\Ax:x\in \conn_\Mt(x_1,\dots,x_n)\}.$$
\end{itemize}

\end{theorem}
\proof{We prove, in turn, that $\Mt^\flat_\Ax$ is a rexpansion of $\Mt$, and then the soundness and completeness of $\Mt^\flat_\Ax$ with respect to $\vdash_\Mt^\Ax$.
\begin{description}
\item[Rexpansion.]
It is easy to see that the PNmatrix $\Mt^\flat_\Ax$ is a refinement of the expansion of $\Mt$ with $\mathcal{E}(x)=\{x\}\times L_\Sigma(P)$. 
$\widetilde{\mathcal{E}}:V^\flat_\Ax\to V$ is such that $\widetilde{\mathcal{E}}(x,A)=x$, and clearly preserves and reflects designated values. Using Proposition~\ref{proofbyvals}, it suffices to show that $\{\widetilde{\mathcal{E}}\circ v^\flat:v^\flat\in\Val_{\Mt^\flat_\Ax}\}=\Val_\Mt^\Ax$.
\end{description}

\noindent Note that if $v^\flat\in\Val_{\Mt^\flat_\Ax}$ and $v^\flat(A)=(x,B)$ then $B\in A^{\inst}$. Namely, we have  $B=A^\sigma$ where $\sigma\in L_\Sigma(P)^P$ is such that 
$\sigma(p)=C$ if $v^\flat(p)=(y,C)$.

\begin{description}
\item[Soundness.]
Since $\Mt^\flat_\Ax$ is a rexpansion of $\Mt$ with $\mathcal{E}$, we know that if $v^\flat\in\Val_{\Mt^\flat_\Ax}$ then $\widetilde{\mathcal{E}}\circ v^\flat\in\Val_\Mt$. 
Further, if $A\in\Ax^{\inst}$ and $v^\flat(A)=(x,B)$ then $B\in (\Ax^{\inst})^{\inst}=\Ax^{\inst}$ and $\widetilde{\mathcal{E}}(v^\flat(A))=x\in D$. We conclude that $\{\widetilde{\mathcal{E}}\circ v^\flat:v^\flat\in\Val_{\Mt^\flat_\Ax}\}\subseteq\Val_\Mt^\Ax$ and thus that $\vdash_\Mt^\Ax\subseteq\vdash_{\Mt^\flat_\Ax}$.

\item[Completeness.]
Reciprocally, if $v\in\Val_\Mt$ and $v(\Ax^{\inst})\subseteq D$ then $v=\widetilde{\mathcal{E}}\circ v^\flat$ with $v^\flat(A)=(v(A),A)$ for each $A\in L_\Sigma(P)$.
Since $v\in\Val_\Mt$, the fact that $v(\Ax^{\inst})\subseteq D$ guarantees that $v^\flat\in\Val_{\Mt^\flat_\Ax}$. We conclude that $\Val_\Mt^\Ax\subseteq
\{\widetilde{\mathcal{E}}\circ v^\flat:v^\flat\in\Val_{\Mt^\flat_\Ax}\}$ and thus that $\vdash_{\Mt^\flat_\Ax}{\subseteq}\vdash_\Mt^\Ax$.
\qed
\end{description}
}

In the definition of $\Mt^\flat_\Ax$, if $\conn_\Mt(x_1,\dots,x_k)\cap D=\emptyset$ and moreover one has $\conn(A_1,\dots,A_k)\in\Ax^{\inst}$ then $\conn_{\Mt^\flat_\Ax}((x_1,A_1),\ldots,(x_k,A_k))=\emptyset$, which in general explains why the resulting PNmatrix may fail to be total. 
Still, $\Mt^\flat_\Ax$ is deterministic (actually a Pmatrix) when $\Mt$ is a (P)matrix. These two observations mean that the construction actually uses partiality in a most relevant way, but not non-determinism, which is simply imported from the starting PNmatrix.
Note also that the construction, though fully illustrative of the power of rexpansions (generalized to PNmatrices) to accomodate new axioms, has other drawbacks. In fact, $\Mt^\flat_\Ax$ is always infinite, even if starting from a finite $\Mt$. Further, the structure of $\Mt^\flat_\Ax$ is quite syntactic, as it incorporates an obvious pattern-matching mechanism for recognizing instances of axioms into the received structure of $\Mt$.\\

In general, it is not possible to do much better, as it may happen that $\vdash_\Mt^\Ax$ cannot be characterized by a finite PNmatrix. 
For instance, as noted in~\cite{rexpansions}, Avron and coauthors show in~\cite{avron2007} that the logic resulting from strengthening the Nmatrix characterizing the basic paraconsistent logic 
$\mathcal{BK}$ of~\cite{Avron2012atLICS} with the axiom $\neg(p_1\e\neg p_1)\to\circ p_1$ yields a logic that cannot be characterized by a finite Nmatrix. 
Thus, in order to improve on our result, it can be useful to look for suitable ways of controlling the shape of the axioms considered, as many other examples are known to have finite characterizations~\cite{avron2005,avron2007,Avron2012atLICS,taming,rexpansions,carcon}.\smallskip

On the other hand, the construction of Theorem~\ref{flat} unveils a very interesting property of PNmatrices: every axiomatic extension of the logic of a finite (or denumerable) PNmatrix can be characterized by a denumerable PNmatrix. Just by itself, the result entails that \emph{intuitionistic propositional logic ($\mathcal{IPL}$)}
can be given by a single denumerable PNmatrix, sharply contrasting with the known fact that a characteristic matrix for $\mathcal{IPL}$ needs to be non-denumerable 
(see~\cite{godel,wronski74,Woj}).  %
\begin{example}\em
Fix a suitable signature containing the two-place connective $\to$, and use the method above for strengthening with the usual axioms $\Int$ of intuitionistic logic the consequence relation characterized by the Nmatrix $\MPt=\tuple{\{0,1\},\{1\},\cdot_\MPt}$ where $\conn_\MPt(x_1,\dots,x_k)=\{0,1\}$ for every $k$-place $\conn\in\Sigma$ such that ${\conn}\,{\neq}\to$, and $\to_\MPt$ has the truth-table below\footnote{For simplicity, in this and other examples, we omit the usual brackets of set notation when describing the truth-tables.}.
\begin{center}
 \begin{tabular}{c | c c c }
${\to_\MPt}$ & $0$ & $1$  \\
\hline
$0$&  $ 0,1 $ & $ 0,1 $ \\
$1$ &$0$ & $ 0,1 $ 
\end{tabular}
\end{center}
It is easy to see that $\vdash_{\MPt}$ is precisely the consequence determined by the single rule $\frac{\;p\quad p\to q\;}{q}$ of \emph{modus ponens}, and so $\vdash_{\MPt^\flat_\Int}$ is precisely $\mathcal{IPL}$.\hfill $\triangle$
\end{example}

This idea applies also to \emph{propositional normal (global) modal logic} $\mathcal{K}$. 
\begin{example}\em
For simplicity, take a signature containing only the $1$-place modality $\square$, and the $2$-place connective $\to$. The logic determined by the rules of \emph{modus ponens} and \emph{necessitation}, i.e., $\frac{p}{\;\square\,p\;}$, is easily seen to be characterized by the Nmatrix $\MPt_\square=\tuple{\{0,1\},\{1\},\cdot_{\MPt_\square}}$ given by the truth-tables below.
\begin{center}
 \begin{tabular}{c | c c c }
${\to_{\MPt_\square}}$ & $0$ & $1$  \\
\hline
$0$&  $ 0,1 $ & $ 0,1 $ \\
$1$ &$0$ & $ 0,1 $ 
\end{tabular}
\qquad\qquad
\begin{tabular}{c | c }
 & ${\square_{\MPt_\square}}$ \\
\hline
$0$& $ 0,1 $ \\
$1$ &$ 1 $ 
\end{tabular}
\end{center}
Collecting in $\Norm$ the usual axioms of classical implication plus the \emph{normalization axiom} $\square(p\to q)\to(\square p\to \square q)$ and applying Theorem~\ref{flat}, we get a denumerable PNmatrix $(\MPt_\square)^\flat_\Norm$ characterizing $\mathcal{K}$.
\hfill $\triangle$
\end{example}

These cases suggest another possible obstacle to improving our result, namely when the received PNmatrix is not-deterministic and actually mixes designated with undesignated values in some entry of its truth-tables. When the basis is deterministic (enough) many examples are known to be finitely characterizable. 

\subsection{A better (less general) construction}\label{sec32}

In order to improve on the construction presented in the previous subsection, we will borrow full inspiration from the construction in~\cite{taming}, and try to push the boundaries of the scope of application of the underlying ideas.

Let $\Sigma$ be a signature, fix $\Sigma^d\subseteq\Sigma$ and set $\mathcal{U}\subseteq (\Sigma\setminus\Sigma^d)^{(1)}$ to be the set of all 
1-place connectives not in $\Sigma^d$. We shall consider the set $\mathcal{U}^*$ of all finite strings of elements of $\mathcal{U}$ (the Kleene closure of $\mathcal{U}$). We shall use $\varepsilon$ to denote the \emph{empty string}, and $uw\in\mathcal{U}^*$ to denote the \emph{concatenation} of strings $u,w\in\mathcal{U}^*$. We use $\prfx(w)$ to denote the set of all \emph{prefixes} of string $w$, including $\varepsilon$.
Given $w\in\mathcal{U}^*$ and $A\in L_\Sigma(P)$ we will use $wA$ to denote the formula defined inductively by $\varepsilon A=A$, and $\bullet wA=\bullet(wA)$ if $\bullet\in\mathcal{U}$.\\

\begin{definition}\label{simpleaxiom}
Let $\conn\in\Sigma$ be a $k$-place connective. \emph{$\Sigma^d$-simple formulas based on $\conn$} are formulas  $B\in L_\Sigma(\{p_1,\dots,p_k\})$ such that $B=A^\sigma$
for some \emph{structure formula} $A\in L_{\Sigma^d}(\{q_1,\dots,q_n,r_1,\dots,r_m\})$ and some substitution $\sigma$ for which:
\begin{itemize} 
\item $\sigma(q_i)=w_ip_j$ with $w_i\in\mathcal{U}^*$ and $1\leq j\leq k$, for each $1\leq i\leq n$, and 
\item $\sigma(r_l)=u_l\conn(p_1,\dots,p_k)$ with $u_l\in\mathcal{U}^*$, for each $1\leq l\leq m$.
\end{itemize}

\noindent For ease of notation, we will simply write
$$A(\dots w_ip_j \dots u_l\conn(p_1,\dots,p_k)\dots)$$
\noindent for a generic $\Sigma^d$-simple formula based on $\conn$.\smallskip


The \emph{look-ahead} set induced by $B$ is
$\Theta_{B}=(\cup_{i=1}^n\prfx(w_i))\cup(\cup_{l=1}^m \prfx(u_l))$.\smallskip

We call $\Sigma^d$-simple formula to any formula which is $\Sigma^d$-simple based on some\footnote{
Since not all the variables $q_1,\dots,q_n,r_1,\dots,r_m$ need to occur in $A$, it may well happen that the subformula $\conn(p_1,\dots,p_k)$ ends up not appearing in the $\Sigma^d$-simple formula $B$ based on $\conn$. For this reason, such a $\Sigma^d$-simple formula can also be based on any available $k'$-place connective distinct from $\conn$, as long as $k'\geq k$ (more precisely, $k'$ needs to be at least as big as the number of distinct variables $p_j$ occurring in $B$).} connective of $\Sigma$.
The \emph{look-ahead} set induced by a set $\Gamma$ of $\Sigma^d$-simple formulas is\footnote{Note that, in our definition, $\Theta_\Gamma$ is not simply the union of the look-ahead sets of each formula in $\Gamma$. We not only want $\Theta_\Gamma$ to be closed for taking prefixes, but we want $\varepsilon\in\Theta_\Gamma$ even if $\Gamma=\emptyset$ (a rather pathological case).
} 
$\Theta_\Gamma=\{\varepsilon\}\cup(\cup_{B\in\Gamma}\Theta_B)$.\hfill $\triangle$
%
%
%
\end{definition}

$\Sigma^d$-simple formulas will be the allowed shapes of our (schema) axioms. Comparing with~\cite{taming}, our setup is strictly more general in that it allows for an arbitrary base signature $\Sigma$. If we set $\Sigma^d$ to consist of the usual 2-place connectives of positive logic $\e,\ou,\to$, and let 
$\Sigma=\Sigma^d\cup\mathcal{U}$ where $\mathcal{U}$ collects a number of additional 1-place connectives (e.g., $\neg,\circ$), we recover the setup of~\cite{taming}. 

For instance, axiom $B=\circ\neg(p_1\e p_2)\to(\neg\circ p_1\ou \neg\circ p_2)$ is $\Sigma^d$-simple in this setting, as can be seen by taking $A=r_1\to(q_1\ou q_2)$, $\conn=\e$, and $\sigma(q_1)=w_1p_1=\neg\circ p_1$, $\sigma(q_2)=w_2p_2=\neg\circ p_2$, thus with $w_1=w_2=\neg\circ$, and $\sigma(r_1)=u_1(p_1\e p_2)=\circ\neg(p_1\e p_2)$, thus with $u_1=\circ\neg$.

Easily, all axioms covered in~\cite{taming} are $\Sigma^d$-simple. However, $p_1\e\neg p_1$ or $p_1\to(\neg p_1\to \neg p_2)$ fall outside the scope of~\cite{taming}, but are still $\Sigma^d$-simple (based on any of the $2$-place connectives, as the $r_l$ variables are not necessary). Axioms like $\neg(p_1\e\neg p_1)\to\circ p_1$ are not $\Sigma^d$-simple, due to the interleaved nesting of $\neg$ and $\e$, and fall outside the scope of both methods.\smallskip

Having set up our syntactic restriction on the set of allowed axioms, we will still need to match them with appropriate semantic restrictions. 
Before we do it, we need to shape up another crucial idea from~\cite{taming}: when strengthening with a set of axioms $\Ax$, the truth-values of the intended PNmatrix will correspond to suitable functions $f:\Theta_\Ax\to V$ where $V$ is the set of truth-values of the given PNmatrix; when the value of a formula $A$ is $f$ this does not only  settle its face value to $f(\varepsilon)$ but also gives as look-ahead information the value $f(w)$ for the value of formulas $wA$ with $w\in\Theta_\Ax $.

\begin{definition}\label{fofv}
Let $\Mt=\tuple{V,D,\cdot_\Mt}$ be a $\Sigma$-PNmatrix and $\Ax$ a set of $\Sigma^d$-simple formulas.
For each $v\in\Val_\Mt$ and $A\in L_\Sigma(P)$, we define $f^A_v\in V^{\Theta_\Ax}$ by letting $f^A_v(w)=v(wA)$ for each $w\in\Theta_\Ax$.\hfill $\triangle$
\end{definition}

It is worth noting that, by definition, $f^A_v(uw)=f^{{w\!A}}_v(u)$ whenever $uw\in\Theta_\Ax$.\\

We can finally put forth our improved construction, taking $\Sigma^d$-simple axioms. In order to make it work it will suffice to require that the given PNmatrix is $\Sigma^d$-deterministic (not necessarily $\Sigma^d$-total). The more general condition, though, will be to require that the PNmatrix is a rexpansion of a $\Sigma^d$-deterministic PNmatrix, as the crucial necessary property is granted by Lemma~\ref{easylemma}.

\begin{theorem}\label{theconstruction}
Let $\Mt=\tuple{V,D,\cdot_\Mt}$ be a $\Sigma$-PNmatrix and $\Ax\subseteq L_\Sigma(P)$.

If there exists $\Sigma^d\subseteq\Sigma$ such that $\Mt$ is a rexpansion of some $\Sigma^d$-deterministic PNmatrix, and the formulas in $\Ax$ are all $\Sigma^d$-simple, then the consequence $\vdash_\Mt^\Ax$ is characterized by the rexpansion
 $\Mt^\sharp_\Ax=\tuple{V^\sharp_\Ax,D^\sharp_\Ax,\cdot_{\Mt^\sharp_\Ax}}$ of $\Mt$ defined by:	
\begin{itemize}
\item $V^\sharp_\Ax=\bigcup\limits_{v\in\Val_{\Mt}^\Ax}\{f^A_v:A\in L_\Sigma(P)\}$,
\item $D^\sharp_\Ax=\{f\in V^\sharp_\Ax:f(\varepsilon)\in D\}$,
\item for each $k\in\nats$ and $\conn\in\Sigma^{(k)}$, $$\conn_{\Mt^\sharp_\Ax}(f_1,\dots,f_k)=$$
$$\bigcup\limits_{v\in\Val_{\Mt}^\Ax}\{f^{\conn(A_1,\dots,A_k)}_v:A_i\in L_\Sigma(P)\textrm{ with }f^{A_i}_v=f_i \textrm{ for } 1\leq i\leq k\}.$$
\end{itemize}

\end{theorem}
\proof{We prove that $\Mt^\sharp_\Ax$ is a rexpansion of $\Mt$, and then its soundness and completeness with respect to $\vdash_\Mt^\Ax$.
\begin{description}
\item[Rexpansion.]
It is simple to check that the PNmatrix $\Mt^\sharp_\Ax$ is a refinement of the expansion of $\Mt$ with $\mathcal{E}(x)=\{f\in V^{\Theta_\Ax}:f(\varepsilon)=x\}$. Just note that one has $f_v^{\conn(A_1,\dots,A_k)}(\varepsilon)=v(\conn(A_1,\dots,A_k))\in\conn_\Mt(v(A_1),\dots,v(A_k))=\conn_\Mt(f_v^{A_1}(\varepsilon),\dots,f_v^{A_k}(\varepsilon))$ whenever it is the case that  $v\in\Val_\Mt$, $k\in\nats$, $\conn\in\Sigma^{(k)}$ and $A_1,\dots,A_k\in L_\Sigma(P)$.
$\widetilde{\mathcal{E}}:V^\sharp_\Ax\to V$ is such that $\widetilde{\mathcal{E}}(f)=f(\varepsilon)$, and clearly preserves and reflects designated values. As before, using Proposition~\ref{proofbyvals}, it suffices to show that $\{\widetilde{\mathcal{E}}\circ v^\sharp:v^\sharp\in\Val_{\Mt^\sharp_\Ax}\}=\Val_\Mt^\Ax$.
\end{description}
\noindent
For a 1-place connective $\bullet\in\mathcal{U}$ and $u\in\mathcal{U}^*$ such that $u\bullet\in\Theta_\Ax$, given a valuation $v^\sharp\in\Val_{\Mt^\sharp_\Ax}$, we have that $v^\sharp(\bullet A)(u)=v^\sharp(A)(u\bullet)$, simply because $v^\sharp(\bullet A)\in\bullet_{\Mt^\sharp_\Ax}(v^\sharp(A))$ and by definition of $\bullet_{\Mt^\sharp_\Ax}$ there must exist $v\in\Val_\Mt^\Ax$ such that $v^\sharp(\bullet A)=f_v^{\bullet B}$ and $v^\sharp(A)=f_v^{B}$. It easily follows, by induction, that if $w\in\mathcal{U}^*$ is such that $uw\in\Theta_\Ax$ then also $v^\sharp(wA)(u)=v^\sharp(A)(uw)$.
%
%
%
%
%
\begin{description}
\item[Soundness.]
Since $\Mt^\sharp_\Ax$ is a rexpansion of $\Mt$ with $\mathcal{E}$, if $v^\sharp\in\Val_{\Mt^\sharp_\Ax}$ then $\widetilde{\mathcal{E}}\circ v^\sharp\in\Val_\Mt$. Hence, when $B=A(\dots w_i A_j \dots u_n \conn(A_1,\dots,A_k)\dots)\in\Ax^{\inst}$ then 
setting $y=v^\sharp(B)(\varepsilon)$ we have $$y\in A_\Mt(\dots v^\sharp(w_i A_j)(\varepsilon) \dots v^\sharp(u_n \conn(A_1,\dots,A_k))(\varepsilon)\dots)=$$ $$A_\Mt(\dots v^\sharp(A_j)(w_i) \dots v^\sharp(\conn(A_1,\dots,A_k))(u_n)\dots).$$
By definition of $\conn_{\Mt^\sharp_\Ax}$, we know there exist $v\in\Val_\Mt^\Ax$ and $B_1,\dots,B_k\in L_\Sigma(P)$ such that 
$v^\sharp(\conn(A_1,\dots,A_k))=f_v^{\conn(B_1,\dots,B_k)}$ and $v^\sharp(A_j)=f_v^{B_j}$ for $1\leq j\leq k$. 
Thus, we have $$y\in A_\Mt(\dots f_v^{B_j}(w_i) \dots f_v^{\conn(B_1,\dots,B_k)}(u_n)\dots)=$$ $$A_\Mt(\dots v(w_i B_j) \dots v(u_n {\conn(B_1,\dots,B_k)})\dots).$$ Clearly, setting $z=v(A(\dots w_i B_j \dots u_n \conn(B_1,\dots,B_k)\dots))$ we also have $$z\in A_\Mt(\dots v(w_i B_j) \dots v(u_n {\conn(B_1,\dots,B_k)})\dots).$$ Using Lemma~\ref{easylemma}, since $A\in L_{\Sigma^d}(P)$ and $\Mt$ is a rexpansion of a $\Sigma^d$-deterministic PNmatrix, we conclude that $y\in D$ iff $z\in D$. Now, it is also the case that $A(\dots w_i B_j \dots u_n \conn(B_1,\dots,B_k)\dots)\in\Ax^{\inst}$ and we know that $v\in\Val_\Mt^\Ax$, so we conclude that $z\in D$. Therefore, $y\in D$ and $v^\sharp(B)\in D^\sharp_\Ax$. We conclude $\{\widetilde{\mathcal{E}}\circ v^\sharp:v^\sharp\in\Val_{\Mt^\sharp_\Ax}\}\subseteq\Val_\Mt^\Ax$ and $\vdash_\Mt^\Ax\subseteq\vdash_{\Mt^\sharp_\Ax}$.
%
%
%
%
\item[Completeness.]
Reciprocally, if $v\in\Val_\Mt^\Ax$ then $v=\widetilde{\mathcal{E}}\circ v^\sharp$ with $v^\sharp(A)=f^A_v$ for each $A\in L_\Sigma(P)$.
It is immediate, by definition of $\conn_{\Mt^\sharp_\Ax}$, that $f_v^{\conn(A_1,\dots,A_k)}\in \conn_{\Mt^\sharp_\Ax}(f_v^{A_1},\dots,f_v^{A_k})$ for every $k$-place connective $\conn\in\Sigma$ and formulas $A_1,\dots,A_k\in L_\Sigma(P)$. We conclude that $v^\sharp\in\Val_{\Mt^\sharp_\Ax}$. 
Therefore, we have $\Val_\Mt^\Ax\subseteq
\{\widetilde{\mathcal{E}}\circ v^\flat:v^\flat\in\Val_{\Mt^\flat_\Ax}\}$ and  $\vdash_{\Mt^\flat_\Ax}\subseteq\vdash_\Mt^\Ax$.\qed
\end{description}}

As intended, we have pushed the boundaries of the method in~\cite{taming} as much as we could.
Beyond the arbitrariness of the signature, and the more permissive syntactic restrictions on the axioms, we also allow a more general PNmatrix to start with. Instead of demanding it to be the two-valued Boolean matrix on the $\Sigma^d$-connectives, we simply require that it be a rexpansion of any Pmatrix. This has the advantage of applying to a large range of non-classical base logics, but also of making the method incremental, allowing us to add axioms one by one and not necessarily all at once. Further, in our method, the interpretation of the connectives not in $\Sigma^d$ is completely unrestricted, which constrasts with~\cite{taming}, where the remaining (1-place) connectives are implicitly forced to be fully non-deterministic.
%
This additional degree of freedom allowed by our method applies not only to the connectives in $\mathcal{U}$, but also to any other connectives not appearing in the structure formulas of the axioms.

\section{Worked examples}\label{sec4}

In order to show the workings and scope of the method we have put forth in Subsection~\ref{sec32}, we shall now consider a few meaningful illustrative examples. 

\begin{example}\label{pnegpq}
\em
Suppose that we want to add to the logic of classical implication a negation connective satisfying the \emph{explosion} axiom $p_1\to(\neg p_1\to p_2)$.

We consider the signature $\Sigma$ with a single 2-place connective $\to$, and a single 1-place connective $\neg$, and we start from the two-valued (P)Nmatrix $\BMt=\tuple{\{0,1\},\{1\},\cdot_\BMt}$ given by the truth-tables below.

\begin{center}
 \begin{tabular}{c | c c c }
${\to_\BMt}$ & $0$ & $1$  \\
\hline
$0$&  $ 1 $ & $ 1 $ \\
$1$ &$0$ & $ 1 $ 
\end{tabular}
\qquad\qquad
\begin{tabular}{c | c }
 & ${\neg_\BMt}$ \\
\hline
$0$& $ 0,1 $ \\
$1$ &$ 0,1 $ 
\end{tabular}
\end{center}
 
Clearly, $\to_\BMt$ corresponds to the usual matrix truth-table of classical implication. The truth-table of $\neg_\BMt$ is fully non-deterministic.

Setting $\Sigma^d$ to contain only $\to$, and $\mathcal{U}=\{\neg\}$ it is clear that $\BMt$ is $\Sigma^d$-deterministic and that the axiom is $\Sigma^d$-simple. With $\Exp=\{p_1\to(\neg p_1\to p_2)\}$, we have that $\Theta_\Exp=\{\varepsilon,\neg\}$. From Theorem~\ref{theconstruction}, the strengthening of $\vdash_\BMt$ with $\Exp$ is characterized by the PNmatrix 
$\BMt^\sharp_\Exp=\tuple{\{00,01,10,11\},\{10,11\},\cdot_{\BMt^\sharp_\Exp}}$ where:
 
 \begin{center}
 \begin{tabular}{c | c c c c}
${\to_{\BMt^\sharp_\Exp}}$ & $00$ & $01$ & $10$ & $11$  \\[2mm]
\hline
$00$&  $ 10 $ & $ 10 $&  $ 10 $ & $ \emptyset $ \\
$01$ &$10$ & $ 10 $&  $ 10 $ & $ \emptyset $ \\
$10$&  $ 00,01 $ & $ 00,01 $&  $ 10 $ & $ \emptyset $ \\
$11$ &$\emptyset$ & $ \emptyset $&  $ \emptyset $ & $ 11 $ 
\end{tabular}
\qquad\qquad
\begin{tabular}{c | c }
 & ${\neg_{\BMt^\sharp_\Exp}}$ \\[2mm]
\hline
$00$& $ 00,01 $ \\
$01$ &$ 10 $ \\
$10$& $ 00,01 $ \\
$11$ &$ 11 $ 
\end{tabular}
\end{center}

Note that, for ease of notation, we are denoting a function $f\in V^\sharp_{\Exp}$ simply by the string $f(\varepsilon)f(\neg)$. For instance, the value $01$ corresponds to the function such that $f(\varepsilon)=0$ and $f(\neg)=1$. In this example, all four possibilities correspond to truth-values of the resulting PNmatrix. The reader may refer to Example~\ref{clun} below, for a situation where this does not happen.\smallskip

For illustration purposes, let us clarify why $\neg_{\BMt^\sharp_\Exp}(10)=\{00,01\}$. Easily, if $xy\in\neg_{\BMt^\sharp_\Exp}(10)$ it is clear that $x=0$ as this is the value of the $\neg$ look-ahead provided by the value $10$. The fact that $y$ can be either $0$ or $1$ boils down to noting that $\neg_\BMt(0)=\{0,1\}$,  none of these choices being incompatible with satisfying the axiom. Namely, $10=f^{p_1}_v$ and $00=f^{\neg p_1}_v$ for any $\BMt$-valuation $v$ with $v(p_1)=1$ and $v(\neg p_1)=v(\neg\neg p_1)=0$ and classical for other formulas, whereas $10=f^{p_1}_v$ and $01=f^{\neg p_1}_v$ would result from any fully classical $\BMt$-valuation with $v(p_1)=1$, both valuations clearly in $\Val_{\BMt}^\Exp$.
Another interesting case is $\neg_{\BMt^\sharp_\Exp}(01)=\{10\}$. Easily, the $1$ on the left of $10$ is explained by the $1$ on the right of $01$. Once again, $\neg_\BMt(1)=\{0,1\}$. However, we must exclude $11$ because $01=f^A_v$ and $11=f^{\neg A}_v$ would jointly imply that $v(\neg A\to(\neg\neg A\to A))=0$ and therefore $v\notin\Val_{\BMt}^\Exp$. Similar justifications can be given, for instance, to explain why $11 \to_{\BMt^\sharp_\Exp}00=\emptyset$.

The PNmatrix 
$\BMt^\sharp_\Exp$ obtained is slightly more complex than one could expect. Note, however, that the value $11$ is isolated from the others in the sense that a valuation that assigns $11$ to some formula must assign $11$ to all formulas. 
Concretely, $\BMt^\sharp_\Exp$ has two maximal total refinements: the three-valued Nmatrix $(\BMt^\sharp_\Exp)_{\{00,01,10\}}$ one would expect, plus the trivial one-valued matrix $(\BMt^\sharp_\Exp)_{\{11\}}$ (whose only trivial valuation is irrelevant for the definition of $\vdash_\BMt^\Exp$).
\hfill$\triangle$
\end{example}

Let us now consider a slight variation on this theme.

\begin{example}\label{pnegpnegq}\em
To see the contrast with the previous example, suppose now that we want to add to the logic of classical implication a negation connective satisfying the weaker \emph{partial explosion} axiom $p_1\to(\neg p_1\to \neg p_2)$. This is a case that is out of the scope of the method in~\cite{taming}.

The setting up we need to consider is the same used in Example~\ref{pnegpq}: the same $\Sigma$, $\Sigma^d$ and $\mathcal{U}$, and the same starting PNmatrix $\BMt$. Setting now $\Exp_\neg=\{p_1\to(\neg p_1\to \neg p_2)\}$, we still have that $\Theta_{\Exp_\neg}=\{\varepsilon,\neg\}$. From Theorem~\ref{theconstruction}, the strengthening of $\vdash_\BMt$ with $\Exp_\neg$ is now characterized by the PNmatrix 
$\BMt^\sharp_{\Exp_\neg}=\tuple{\{00,01,10,11\},\{10,11\},\cdot_{\BMt^\sharp_{\Exp_\neg}}}$ where, using the same notation convention used in Example~\ref{pnegpq}, we have: 

 \begin{center}
 \begin{tabular}{c | c c c c}
${\to_{\BMt^\sharp_{\Exp_\neg}}}$ & $00$ & $01$ & $10$ & $11$  \\[2mm]
\hline
$00$&  $ 10 $ & $ 10 $&  $ 10 $ & $ \emptyset $ \\
$01$ &$10$ & $ 10,11 $&  $ 10 $ & $ 11 $ \\
$10$&  $ 00,01 $ & $ 00,01 $&  $ 10 $ & $ \emptyset $ \\
$11$ &$\emptyset$ & $ 01 $&  $ \emptyset $ & $ 11 $ 
\end{tabular}
\qquad\qquad
\begin{tabular}{c | c }
 & ${\neg_{\BMt^\sharp_{\Exp_\neg}}}$ \\[2mm]
\hline
$00$& $ 00,01 $ \\
$01$ &$ 10,11 $ \\
$10$& $ 00,01 $ \\
$11$ &$ 11 $ 
\end{tabular}
\end{center}

The PNmatrix 
$\BMt^\sharp_{\Exp_\neg}$ is more interesting than before. Note that it also has two maximal total refinements: the three-valued Nmatrix $(\BMt^\sharp_{\Exp_\neg})_{\{00,01,10\}}$ (which is precisely the same as the one obtained in Example~\ref{pnegpq}), plus the two-valued matrix $(\BMt^\sharp_{\Exp_\neg})_{\{01,11\}}$ (whose implication is classical but whose negation is always designated).
\hfill$\triangle$
\end{example}

Next, we will analyze a number of examples that appear scattered in the literature, and show how our method can be systematically used in all of them. 
We start by revisiting an example from~\cite{avronnegations}, paradigmatic of many similar examples considered by Avron and coauthors.

\begin{example}\label{clun}
\em
Let us consider strengthening the logic $\mathcal{CL}u\mathcal{N}$ from~\cite{batens1,batens2} with the \emph{double negation elimination} axiom $\neg\neg p_1\to p_1$. Actually, for the sake of simplicity, we shall consider only the $\{\neg,\to\}$-fragment of the logic.

Let $\Sigma_d$ contain a single 2-place connective $\to$, $\mathcal{U}$ contain a 1-place connective $\neg$. The (fragment of the) logic $\mathcal{CL}u\mathcal{N}$ is characterized by the Nmatrix $\Mt=\tuple{\{0,1\},\{1\},\cdot_\Mt}$ with:
 
 \begin{center}
 \begin{tabular}{c | c c c }
${\to_\Mt}$ & $0$ & $1$  \\
\hline
$0$&  $ 1 $ & $ 1 $ \\
$1$ &$0$ & $ 1 $ 
\end{tabular}
\qquad\qquad
\begin{tabular}{c | c }
 & ${\neg_\Mt}$ \\
\hline
$0$& $ 1 $ \\
$1$ &$ 0,1 $ 
\end{tabular}
\end{center}
 
It is clear that $\Mt$ is $\Sigma_d$-deterministic and that the axiom is $\Sigma^d$-simple. If we let $\mathsf{DNe}=\{\neg\neg p_1\to p_1\}$, we have 
that $\Theta_{\mathsf{DNe}}=\{\varepsilon,\neg,\neg\neg\}$. From Theorem~\ref{theconstruction}, the strengthening of $\vdash_\Mt$ with $\mathsf{DNe}$, 
which is well known to coincide with the logic $\mathcal{C}_{\min}$ of \cite{limits,taxonomy}, is characterized by the four-valued Nmatrix 
$\Mt^\sharp_{\mathsf{DNe}}=\tuple{\{010,101,110,111\},\{101,110,111\},\cdot_{\Mt^\sharp_{\mathsf{DNe}}}}$ where:

%
%
 
 \begin{center}
 \begingroup
\renewcommand{\arraystretch}{1.2} 
 \begin{tabular}{c | c c c c}
${\to_{\Mt^\sharp_{\mathsf{DNe}}}}$ & $010$ & $101$ & $110$ & $111$  \\[2mm]
\hline
$010$&  $ D^\sharp_{\mathsf{DNe}} $ & $ D^\sharp_{\mathsf{DNe}} $&  $ D^\sharp_{\mathsf{DNe}} $ & $ D^\sharp_{\mathsf{DNe}} $ \\
$101$ &$010$ & $ D^\sharp_{\mathsf{DNe}}$&  $ D^\sharp_{\mathsf{DNe}} $ & $ D^\sharp_{\mathsf{DNe}} $ \\
$110$&  $ 010 $ & $D^\sharp_{\mathsf{DNe}}$&  $ D^\sharp_{\mathsf{DNe}} $ & $ D^\sharp_{\mathsf{DNe}} $ \\
$111$ &$010$ & $D^\sharp_{\mathsf{DNe}} $&  $D^\sharp_{\mathsf{DNe}} $ & $ D^\sharp_{\mathsf{DNe}} $ 
\end{tabular}
\qquad\qquad
\begin{tabular}{c | c }
 & ${\neg_{\Mt^\sharp_{\mathsf{DNe}}}}$ \\[2mm]
\hline
$010$& $ 101 $ \\
$101$ &$ 010 $ \\
$110$& $ 101 $ \\
$111$ &$ 110,111 $ 
\end{tabular}
\endgroup
\end{center}

Above, for ease of notation, we are denoting a function $f\in V^\sharp_{{\mathsf{DNe}}}$ simply by the string $f(\varepsilon)f(\neg)f(\neg\neg)$.
As this is a new feature in our row of examples, it is worth explaining why only four of the eight possible such functions appear as truth-values of the resulting Nmatrix. Namely, $000,001,100$ are all unattainable as $f_v^A$ in the Nmatrix $\Mt$ since $\neg_{\Mt^\sharp_{\mathsf{DNe}}}(0)=1$. The remaining string $011$ is excluded for more interesting reasons, as any $\Mt$-valuation $v$ with  $001=f_v^A$ makes $v(\neg\neg A\to A)=0$ and thus $v\notin\Val_{\Mt}^{\mathsf{DNe}}$.\smallskip

This example shows that our method, though very general, may not be as tight as possible. It is a mandatory topic for further research to best understand how to equate the equivalence between this Nmatrix and the three-valued Nmatrix from~\cite{avronnegations}.\smallskip

 If we want to strengthen the resulting logic,  $\mathcal{C}_{\min}$, with the \emph{double negation introduction} axiom $p_1\to \neg\neg p_1$, we can readily apply Theorem~\ref{theconstruction} to $\Mt^\sharp_{\mathsf{DNe}}$ and $\mathsf{DNi}=\{ p_1\to \neg\neg p_1\}$, obtaining (up to renaming of the truth-values) the three-valued Nmatrix $\mathbb{N}=(\Mt^\sharp_{\mathsf{DNe}})^\sharp_{\mathsf{DNi}}=\tuple{\{01,10,11\},\{10,11\},\cdot_{\mathbb{N}}}$ where:
 
  \begin{center}
 \begin{tabular}{c | c c c c}
${\to_{\mathbb{N}}}$ & $01$ & $10$ & $11$  \\
\hline
$01$&  $10,11 $ & $10,11 $&  $ 10,11$ \\
$10$ &$01$ & $ 10,11$&   $ 10,11 $ \\
$11$ &$01$ & $10,11 $&   $ 10,11 $ 
\end{tabular}
\qquad\qquad
\begin{tabular}{c | c }
 & ${\neg_{\mathbb{N}}}$ \\
\hline
$01$& $ 10 $ \\
$10$ &$ 01$ \\
$11$ &$ 11 $ 
\end{tabular}
\end{center}

Note that, by construction, the Nmatrix $\mathbb{N}$ has three values $g:\Theta_{\mathsf{DNi}}\to V^\sharp_{\mathsf{DNe}}$ which, 
given that $\Theta_{\mathsf{DNi}}=\{\varepsilon,\neg,\neg\neg\}$, can be written in string notation as $g(\varepsilon)g(\neg)g(\neg\neg)$,  corresponding to the strings $010101010,101010101,111111111$. Clearly, each of them can be named simply by their first two symbols.

It is interesting to further note that this Nmatrix is isomorphic to $\Mt^\sharp_{\mathsf{DNe}\cup\mathsf{DNi}}$. This is a particularly happy case as, in general, adding axioms incrementally, instead of all at once (as in~\cite{taming}), will yield an equivalent PNmatrix but not necessarily the same, often with more truth-values.\hfill$\triangle$
\end{example}

We now consider a more elaborate example in the family of paraconsistent logics,  as also tackled by Avron and coauthors, which is developed in detail in~\cite{taming}.

\begin{example}\label{agata}\em
As in Example 5.1 of~\cite{taming}, we want to characterize the logic obtained by adding two additional 1-place connectives $\neg,\circ$ to positive classical logic, subject to the set of axioms $\Ax$ containing:
$$p_1\ou \neg p_1$$ $$p_1\to(\neg p_1\to(\circ p_1 \to p_2))$$ $$\circ p_1\ou (p_1 \e \neg p_1)$$ $$\circ p_1\to\circ(p_1\e p_2)$$ $$(\neg p_1\ou \neg p_2)\to\neg(p_1\e p_2)$$

  Let $\Sigma_d$ contain the three 2-place connectives $\e,\ou,\to$, and $\mathcal{U}$ contain the two 1-place connectives $\neg,
  \circ$ and consider the Nmatrix  $\CMt=\tuple{\{0,1\},\{1\},\cdot_\CMt}$ with:
 
 \begin{center}
  \begin{tabular}{c | c c c }
${\e_\CMt}$ & $0$ & $1$  \\
\hline
$0$&  $ 0 $ & $ 0 $ \\
$1$ &$0$ & $ 1 $ 
\end{tabular}\qquad
  \begin{tabular}{c | c c c }
${\ou_\CMt}$ & $0$ & $1$  \\
\hline
$0$&  $ 0$ & $ 1 $ \\
$1$ &$1$ & $ 1 $ 
\end{tabular}\qquad
 \begin{tabular}{c | c c c }
${\to_\CMt}$ & $0$ & $1$  \\
\hline
$0$&  $ 1 $ & $ 1 $ \\
$1$ &$0$ & $ 1 $ 
\end{tabular}
\qquad
\begin{tabular}{c | c | c}
 & ${\neg_\CMt}$  & ${\circ_\CMt}$ \\
\hline
$0$& $ 0,1  $&$ 0,1 $ \\
$1$ &$ 0,1 $&$ 0,1 $ 
\end{tabular}
\end{center}

It is clear that $\CMt$ is $\Sigma_d$-deterministic and that the axioms are all $\Sigma^d$-simple. 
Further, we get $\Theta_{\Ax}=\{\varepsilon,\neg,\circ\}$. From Theorem~\ref{theconstruction}, the strengthening $\vdash^\Ax_\CMt$
is characterized by the PNmatrix 
$\CMt^\sharp_{\Ax}=\tuple{\{011,101,110,111\},\{101,110,111\},\cdot_{\CMt^\sharp_{\Ax}}}$ where:

%
%
%
 \begin{center}
  \begin{tabular}{c | c c c c}
${\e_{\CMt^\sharp_\Ax}}$ & $011$ & $101$ & $110$ & $111$  \\[2mm]
\hline
$011$&  $011$ & $011$&   $011$  & $ \emptyset $ \\
$101$ &$011$ & $101$ &  $\emptyset$ & $ \emptyset $ \\
$110$&  $ 011 $ & $\emptyset$ & $110$ & $ \emptyset$ \\
$111$ &$\emptyset$ & $\emptyset $&  $\emptyset $ & $ 111 $ 
\end{tabular}
\qquad
 \begin{tabular}{c | c c c c}
${\ou_{\CMt^\sharp_\Ax}}$ & $011$ & $101$ & $110$ & $111$  \\[2mm]
\hline
$011$&  $011$ & $101$&   $110$  & $ \emptyset $ \\
$101$ &$101$ & $101$ &  $\emptyset$ & $ \emptyset $ \\
$110$&  $110 $ & $\emptyset$ & $110$ & $ \emptyset$ \\
$111$ &$\emptyset$ & $\emptyset $&  $\emptyset $ & $ 111 $ 
\end{tabular}
\end{center}

\begin{center}
 \begin{tabular}{c | c c c c}
${\to_{\CMt^\sharp_\Ax}}$ & $011$ & $101$ & $110$ & $111$  \\[2mm]
\hline
$011$&  $101,110$ & $101$&   $110$  & $ \emptyset $ \\
$101$ &$011$ & $101$ &  $\emptyset$ & $ \emptyset $ \\
$110$&  $ 011 $ & $\emptyset$ & $110$ & $ \emptyset$ \\
$111$ &$\emptyset$ & $\emptyset $&  $\emptyset $ & $ 111 $ 
\end{tabular}
\qquad
\begin{tabular}{c | c | c}
 & ${\neg_{\CMt^\sharp_\Ax}}$ & ${\circ_{\CMt^\sharp_\Ax}}$ \\[2mm]
\hline
$011$& $ 101,110 $ & $ 101,110 $\\
$101$ &$ 011 $ & $ 101$\\
$110$& $ 110 $ & $ 011 $\\
$111$ &$ 111 $ & $ 111 $
\end{tabular}
\end{center}

For ease of notation, once again, we are denoting a function $f\in V^\sharp_{\Ax}$ simply by the string $f(\varepsilon)f(\neg)f(\circ)$.

Notably, the PNmatrix $\CMt^\sharp_\Ax$ is slightly different from the PNmatrix obtained using the method in~\cite{taming}. Still, it is easy to see that $\CMt^\sharp_\Ax$ has two maximal total refinements: the three-valued PNmatrix $(\CMt^\sharp_{\Ax})_{\{011,101,110\}}$ (which is an equivalent refinement of the PNmatrix in~\cite{taming} maximizing the partiality), plus the trivial one-valued matrix $(\CMt^\sharp_{\Ax})_{\{111\}}$.
\hfill$\triangle$

\end{example}

Our next example deals with Nelson-like logics and twist-structures.

\begin{example} \label{nelson}\em

The addition of a paraconsistent Nelson-like~\cite{nelson,vakarelovnelson,odintsovBook} \emph{strong negation} $\sim$  to a given intermediate logic (as in~\cite{Kracht1998OnEO}) can be easily captured by our construction. 

Let $\Sigma_d$ be a signature containing binary connectives $\e,\ou,\to$, and $\mathcal{U}$ contain the 1-place connective $\sim$,
and consider an Nmatrix $\Mt=\tuple{V,\{1\},\cdot_\Mt}$ whose 
$\{\e,\ou,\to\}$-reduct of $\Mt$, dubbed $\mathbb{N}$, 
is an \emph{implicative lattice} \cite{odintsovBook}, and such that $\sim_\Mt(x)=V$ for every $x\in V$, and let $\Ax$ contain:
$$\sim\sim p_1\to p_1\qquad\qquad  p_1\to\,\sim\sim p_1$$
$$\sim(p_1\ou p_2)\to(\sim p_1\e\sim p_2)\qquad\qquad(\sim p_1\,\e\sim p_2)\to\,\sim(p_1\ou p_2)$$
$$\sim(p_1\e p_2)\to(\sim p_1\ou\sim p_2)\qquad\qquad (\sim p_1\ou\sim p_2)\to\, \sim(p_1\e p_2)$$
$$\sim(p_1\to p_2)\to(p_1\e\sim p_2)\qquad\qquad  (p_1\e\sim p_2)\to \,\sim(p_1\to p_2)$$

\vspace*{2mm}
Clearly, the axioms in $\Ax$ are $\Sigma_d$-simple
 and $\Theta_\Ax=\{\varepsilon,\sim,\sim\sim\}$.
From Theorem~\ref{theconstruction}, $\vdash^\Ax_\Mt$ is characterized by the matrix 
$\Mt^\sharp_{\Ax}=\tuple{V^\sharp_\Ax,D^\sharp_\Ax,\cdot_{\Mt^\sharp_{\Ax}}}$ isomorphic to the well known full twist-structure $\mathbb{N}^{\bowtie}$ over $\mathbb{N}$ (see~\cite{odintsovBook}). 
 Namely, we have $V^\sharp_\Ax=\{f\in V^{\{\varepsilon,\sim,\sim\sim\}}: f(\varepsilon)=f(\sim\sim)\}$. For simplicity, we can represent each such function $f\in V^\sharp_\Ax$ simply by the pair $(f(\varepsilon),f(\sim))$. 
 Hence, we have:
\begin{itemize} 
\item $V^\sharp_\Ax=V\times V$ and $D^\sharp_\Ax=\{1\}\times V$,  
\item $(x_1,y_1)\e_{\Mt^\sharp_\Ax}(x_1,y_1)=(x_1\e_\Mt x_2,y_1\ou_\Mt y_2)$, 
\item $(x_1,y_1)\ou_{\Mt^\sharp_\Ax}(x_1,y_1)=(x_1\ou_\Mt x_2,y_1\e_\Mt y_2)$, 
\item $(x_1,y_1)\to_{\Mt^\sharp_\Ax}(x_2,y_2)=(x_1\to_\Mt x_2,x_1\e_\Mt y_2)$, and 
\item $\sim_{\Mt^\sharp_\Ax}({x,y})=(y,x)$. 
\end{itemize}
When we take $\mathbb{N}$ to be the two-valued Boolean matrix, and using now $xy$ instead of $(x,y)$, we obtain, 
$\Mt^\sharp_{\Ax}=\tuple{\{00,01,10,11\},\{10,11\},\cdot_{\Mt^\sharp_{\Ax}}}$ where: 
 \begin{center}
 \begin{tabular}{c | c c c c}
${\e_{\Mt^\sharp_\Ax}}$ & $00$ & $01$ & $10$ & $11$  \\[2mm]
\hline
$00$&  $ 00 $ & $ 01 $&  $ 00 $ & $ 01 $ \\
$01$ &$01$ & $ 01 $&  $ 01 $ & $01$ \\
$10$&  $00$ & $ 01$&  $ 10 $ & $ 11 $ \\
$11$ &$01$ & $ 01 $&  $ 11 $ & $ 11 $ 
\end{tabular}
\quad
 \begin{tabular}{c | c c c c}
${\ou_{\Mt^\sharp_\Ax}}$ & $00$ & $01$ & $10$ & $11$  \\[2mm]
\hline
$00$&  $ 00 $ & $ 00 $&  $ 10 $ & $ 10 $ \\
$01$ &$00$ & $01 $&  $ 10 $ & $ 11 $ \\
$10$&  $ 10 $ & $ 10 $&  $ 10 $ & $ 10 $ \\
$11$ &$10$ & $11 $&  $ 10 $ & $ 11 $ 
\end{tabular}
\end{center}

\begin{center}
  \begin{tabular}{c | c c c c}
${\to_{\Mt^\sharp_\Ax}}$ & $00$ & $01$ & $10$ & $11$  \\[2mm]
\hline
$00$&  $ 10 $ & $ 10 $&  $ 10 $ & $ 10 $ \\
$01$ &$10$ & $ 10 $&  $ 10 $ & $ 10$ \\
$10$&  $ 00 $ & $ 01 $&  $ 10 $ & $11$ \\
$11$ &$00$ & $01$&  $10$ & $11$ 
\end{tabular}\qquad\quad
\begin{tabular}{c | c }
 & ${\sim_{\Mt^\sharp_\Ax}}$ \\[2mm]
\hline
$00$& $ 00 $ \\
$01$ &$ 10 $ \\
$10$& $ 01 $ \\
$11$ &$ 11 $ 
\end{tabular}
\end{center}

Note this semantics coincides precisely with the semantic extension of Belnap's four-valued logic~\cite{Belnap19774val,Belnapcomputerthink} with \emph{true implication} of Avron~\cite{AvronValue4Values}.

If we further impose the axiom $$\sim p_1\to(p_1\to p_2)$$
we obtain corresponding explosive versions of Nelson's construction. 
Making $\Ax'=\Ax\cup\{\sim p_1\to(p_1\to p_2)\}$,
the resulting twist-structure is now a refinement resulting from isolating the truth-value $(1,1)$,
 i.e., such that for $\ast\in\{\e,\ou,\to\}$ we have $(1,1)\ast_{\Mt^\sharp_\Ax}(x,y)=(x,y)\ast_{\Mt^\sharp_\Ax}(1,1)=\emptyset$ if $(x,y)\neq(1,1)$. 
Concretely, if we take $\mathbb{N}$ to be the two-valued Boolean matrix, again, we obtain the Pmatrix
$\Mt^\sharp_{\Ax'}=\tuple{\{00,01,10,11\},\{10,11\},\cdot_{\Mt^\sharp_{\Ax'}}}$ where:

 \begin{center}
 \begin{tabular}{c | c c c c}
${\e_{\Mt^\sharp_{\Ax'}}}$ & $00$ & $01$ & $10$ & $11$  \\[2mm]
\hline
$00$&  $ 00 $ & $ 01 $&  $ 00 $ & $\emptyset $ \\
$01$ &$01$ & $ 01 $&  $ 01 $ & $\emptyset$ \\
$10$&  $00$ & $ 01$&  $ 10 $ & $ \emptyset $ \\
$11$ &$\emptyset$ & $ \emptyset $&  $ \emptyset $ & $ 11 $ 
\end{tabular}
\quad
 \begin{tabular}{c | c c c c}
${\ou_{\Mt^\sharp_{\Ax'}}}$ & $00$ & $01$ & $10$ & $11$  \\[2mm]
\hline
$00$&  $ 00 $ & $ 00 $&  $ 10 $ & $ \emptyset $ \\
$01$ &$00$ & $01 $&  $ 10 $ & $\emptyset $ \\
$10$&  $ 10 $ & $ 10 $&  $ 10 $ & $\emptyset $ \\
$11$ &$\emptyset$ & $\emptyset$&  $\emptyset $ & $ 11 $ 
\end{tabular}
\end{center}

\begin{center}
  \begin{tabular}{c | c c c c}
${\to_{\Mt^\sharp_{\Ax'}}}$ & $00$ & $01$ & $10$ & $11$  \\[2mm]
\hline
$00$&  $ 10 $ & $ 10 $&  $ 10 $ & $\emptyset$ \\
$01$ &$10$ & $ 10 $&  $ 10 $ & $ \emptyset$ \\
$10$&  $ 00 $ & $ 01 $&  $ 10 $ & $\emptyset$ \\
$11$ &$\emptyset$ & $\emptyset$&  $\emptyset$ & $11$ 
\end{tabular}\quad\qquad
\begin{tabular}{c | c }
 & ${\neg_{\Mt^\sharp_{\Ax'}}}$ \\[2mm]
\hline
$00$& $ 00 $ \\
$01$ &$ 10 $ \\
$10$& $ 01 $ \\
$11$ &$ 11 $ 
\end{tabular}
\end{center}

%
%
Easily, $\Mt^\sharp_{\Ax'}$ has two maximal total refinements: the three-valued matrix $(\Mt^\sharp_{\Ax'})_{\{00,01,10\}}$, plus the trivial one-valued matrix $(\Mt^\sharp_{\Ax'})_{\{11\}}$. Expectedly,  we have that
$(\Mt^\sharp_{\Ax'})_{\{00,01,10\}}$ is precisely the matrix characterizing the three-valued logic of Vakarelov~\cite{vakarelovnelson,Kracht1998OnEO} (which coincides with $\vdash^{\Ax'}_\Mt$, and is known to be translationally equivalent to \L ukasiewicz's three-valued logic).
\hfill$\triangle$
\end{example}

Next, we will show, by means of an example, that our method subsumes the idea of \emph{swap-structure semantics} put forth in~\cite{carcon,coniglioswap}.

\begin{example}\label{swap}\em
As in~\cite{coniglioswap}, we consider obtaining a semantic characterization of the non-normal modal logic $\mathcal{T}$ of Kearns~\cite{kearns}, which coincides with the logic $\mathcal{S}_a+$ of Ivlev~\cite{Ivlev}. This can be done by using our method to characterize the logic obtained by a 1-place connective $\square$ to the $\{\neg,\to\}$-fragment of classical logic, further demanding the 
$\Tm$ axioms of~\cite{coniglioswap}, namely:
$$\square(p_1\to p_2)\to (\square p_1 \to \square  p_2)$$
$$ \square(p_1\to p_2)\to (\square \neg p_2 \to \square \neg p_1)$$
$$  \neg\square\neg(p_1\to p_2)\to (\square p_1 \to \neg\square\neg p_2)$$
$$ \square \neg p_1 \to \square (p_1\to p_2)$$
$$ \square p_2 \to \square (p_1\to p_2)$$
$$ \square \neg (p_1\to p_2) \to \square \neg p_2$$
$$ \square \neg (p_1\to p_2) \to \square p_1$$
$$ \square p_1\to p_1$$
$$ \square p_1 \to \square \neg \neg p_1$$
$$\square\neg \neg p_1\to \square p_1 $$

Let $\Sigma_d$ contain $\to$, and $\mathcal{U}=\{\neg,\square\}$. Take the Nmatrix $\DMt=\tuple{\{0,1\},\{1\},\cdot_\DMt}$ with:
 
 \begin{center}
 \begin{tabular}{c | c c c }
${\to_\DMt}$ & $0$ & $1$  \\
\hline
$0$&  $ 1 $ & $ 1 $ \\
$1$ &$0$ & $ 1 $ 
\end{tabular}
\qquad\qquad
\begin{tabular}{c | c |c }
 & ${\neg_\DMt}$& ${\square_\DMt}$ \\
\hline
$0$& $ 1 $& $0,1$ \\
$1$ &$ 0 $ & $0,1$
\end{tabular}
\end{center}

Clearly the axioms in $\Tm$ are $\Sigma_d$-simple.  
Furthermore, now, we have that $\Theta_\Tm=\{\varepsilon\}\cup\prfx(\{\square,\square\neg,\neg\square\neg,\square\neg\neg\})=\{\varepsilon,\neg,\neg\square,\neg\square\neg,\square,\square\neg,\square\neg\neg\}$.
Note that for any $f\in V^\sharp_\Tm$ and $\neg w\in \Theta_\Tm$ we have $f(\neg w)=1-f(w)$. 
Note also that due to the last two axioms of $\Tm$, it follows that $f(\square\neg\neg)=f(\square)$ for any $f\in V^\sharp_\Tm$.
Hence, we can represent each $f$ simply by the string $f(\varepsilon)f(\square)f(\square \neg)$.
Further, note that the antepenultimate axiom $ \square p_1\to p_1$ guarantees both that $f(\square)\leq f(\varepsilon)$ and 
$f(\square \neg)\leq f(\neg)=1-f(\varepsilon)$.  
Now, applying Theorem~\ref{theconstruction}, we conclude that the strengthening $\vdash^\Tm_\DMt$
is characterized by the four-valued Nmatrix given by
$\DMt^\sharp_{\Tm}=\tuple{\{000,001,100,110\},\{100,110\},\cdot_{\DMt^\sharp_{\Tm}}}$ where: 
 
 \begin{center}
  \begin{tabular}{c | c c c c}
${\to_{\DMt^\sharp_{\Tm}}}$ & $000$ & $001$ & $100$ & $110$  \\[2mm]
\hline
$000$&  $100,110$ & $100$&   $100,110$  & $110$ \\
$001$ &$110$ & $110$ &  $110$ & $110$ \\
$100$&  $000$ & $000$ & $100,110$ & $110$ \\
$110$ &$000$ & $001 $&  $100$ &$110$ 
\end{tabular}
\qquad
\begin{tabular}{c | c | c}
 & ${\neg_{\DMt^\sharp_\Tm}}$ & ${\square_{\DMt^\sharp_\Tm}}$ \\[2mm]
\hline
$000$& $ 100 $ & $000,001$\\
$001$ &$ 110 $ & $000,001$\\
$100$& $ 000 $ & $000,001$\\
$110$ &$ 001 $ & $ 100,110$
\end{tabular}
\end{center}
It is straightforward to check that this Nmatrix is isomorphic to the Kearns and Ivlev semantics~\cite{kearns,Ivlev}, also recovered in~\cite{coniglioswap}, by renaming the truth-values $000,001,100,110$ by $f,F,t,T$, respectively.\hfill$\triangle$
\end{example}

We finish this section with another example, starting from a non-classical base, namely, {\L}ukasiewicz's five-valued logic.

\begin{example}\label{L53}\em
We start from {\L}ukasiewicz's logic $\mathcal{L}_5$ and strengthen it by axiom $((p_1\to \neg p_1)\to p_1)\to p_1$ in order to obtain {\L}ukasiewicz's three-valued logic $\mathcal{L}_3$ (see, for instance,~\cite{Woj,gott}). In this case, no new connectives are added.

%
%

Let $\Sigma_d$ contain the 2-place connective $\to$, and also the 1-place connective $\neg$, and let $\mathcal{U}=\emptyset$.
Let also $\Ax=\{((p_1\to \neg p_1)\to p_1)\to p_1\}$.
Consider the five-valued matrix $\mathbb{L}_5=\tuple{\{0,\frac{1}{4},\frac{1}{2},\frac{3}{4},1\},\{1\},\cdot_{\mathbb{L}_5}}$ with:
 
 \begin{center} 
 \begingroup
\renewcommand{\arraystretch}{1.2}
 \begin{tabular}{c | c c c c c}
${\to_{\mathbb{L}_5}}$ & $0$&$\frac{1}{4}$&$\frac{1}{2}$&$\frac{3}{4}$&$1$  \\[1mm]
\hline
$0$ 		    & $1$&   $1$&                $1$&               $1$&             $1$  \\
$\frac{1}{4}$ & $\frac{3}{4}$&$1$&$1$&$1$&$1$  \\
$\frac{1}{2}$& $\frac{1}{2}$&$\frac{3}{4}$&$1$&$1$&$1$  \\
$\frac{3}{4}$& $\frac{1}{4}$&$\frac{1}{2}$&$\frac{3}{4}$&$1$&$1$  \\
$1$              & $0$&$\frac{1}{4}$&$\frac{1}{2}$&$\frac{3}{4}$&$1$  
\end{tabular}\qquad\quad
  \begin{tabular}{c | c }
 & ${\neg_{\mathbb{L}_5}}$   \\[1mm]
\hline
$0$ 		    & $1$\\  
$\frac{1}{4}$ & $\frac{3}{4}$ \\
$\frac{1}{2}$& $\frac{1}{2}$  \\
$\frac{3}{4}$& $\frac{1}{4}$  \\
$1$              & $0$ 
\end{tabular}
\endgroup
\end{center}
Clearly the axiom is $\Sigma_d$-simple and $\Theta_\Ax=\{\varepsilon\}$.
Hence we represent any $f\in V^\sharp_\Ax$  simply by $f(\varepsilon)$.
 From Theorem~\ref{theconstruction}, the strengthening $\vdash^\Ax_{\mathbb{L}_5}$
is characterized by the well-known three-valued matrix
$({\mathbb{L}_5})^\sharp_{\Ax}=\mathbb{L}_3=\tuple{\{0,\frac{1}{2},1\},\{1\},\cdot_{\mathbb{L}_3}}$ where: 
 
 \begin{center} 
 \begin{tabular}{c | c c c c c}
${\to_{\mathbb{L}_3}}$ & $0$&$\frac{1}{2}$&$1$  \\[1mm]
\hline
$0$ 		    & $1$&                 $1$&                       $1$  \\
$\frac{1}{2}$& $\frac{1}{2}$& $1$&$1$  \\
$1$              & $0$&$\frac{1}{2}$&$1$  
\end{tabular}\qquad\quad
  \begin{tabular}{c | c }
 & ${\neg_{\mathbb{L}_3}}$   \\[1mm]
\hline
$0$ 		    & $1$\\  
$\frac{1}{2}$& $\frac{1}{2}$  \\
$1$              & $0$ 
\end{tabular}

\end{center}

%
%
%
%
%
%
\hfill$\triangle$
\end{example}

Examples~\ref{pnegpq},~\ref{agata},~\ref{swap} are also covered by the method in~\cite{taming}. 
The two-valued based case of Example~\ref{nelson} could also be obtained using~\cite{taming}, but not the general case we deal with, over an arbitrary implicative lattice. Example~\ref{clun}, the way it is formulated, is outside the scope of~\cite{taming}, not only because it starts from a Nmatrix where negation is not fully non-deterministic, but also because we are adding one axiom and then another. Examples~\ref{pnegpnegq},~\ref{L53} are also not covered by~\cite{taming}. Namely, Example~\ref{pnegpnegq} uses an axiom which does not respect their syntactic criteria, and Example~\ref{L53} uses a five-valued non-classical matrix.

\section{Analytic multiple-conclusion calculi}\label{sec5}

In the work of Arnon Avron on Nmatrices and rexpansions, obtaining a concise semantics for a logic (typically in the form of a Nmatrix) is not an end in itself but a means for obtaining (sequent-like) analytic calculi for that logic~\cite{AvronBK07,Avron2012atLICS,avroncutfree2013}. In other works (e.g.,~\cite{taming,Baaz2013}), the semantics (typically in the form of a PNmatrix) is not a basis for obtaining a calculus but it is still instrumental in proving its analyticity (when the PNmatrix is total). 
In this paper, so far, we have not worried about proof-theoretic aspects. Therefore, this is a good point for applying to our previous construction the techniques developed in~\cite{synt,wollic19} for obtaining analytic multiple-conclusion calculi for logics defined by finite PNmatrices, under a reasonable expressiveness proviso. This contrasts with the above mentioned results for sequent-like calculi~\cite{avroncutfree2013,Baaz2013,taming}, for which partiality seems to devoid them of a usable (even if generalized) subformula property capable of guaranteeing analyticity (and elimination of non-analytic cuts).\smallskip

In what follows, we will consider so-called \emph{multiple-conclusion calculi}, a simple generalization of Hilbert-style calculi with (schematic) inference rules of the form $\frac{\;\Gamma\;}{\Delta}$ where $\Gamma$ (\emph{premises} read conjunctively, as usual) and $\Delta$ (\emph{conclusions} read disjunctively) are sets of formulas. Such calculi were studied by Shoesmith and Smiley in~\cite{ShoesmithSmiley}, and have very interesting properties. 
A set $R$ of such multiple-conclusion rules induces a consequence relation $\der_R$ by means of an adequate notion of proof, simply defined as a tree-like version of Hilbert-style proofs. We shall show some illustrative examples later, but refer the reader to~\cite{ShoesmithSmiley,synt,wollic} for details. As usual, we say that $R$ constitutes a calculus for a consequence relation $\der$ if $\der_R{=}\der$.

A set $\mathcal{S}\subseteq L_\Sigma(\{p\})$ induces a simple notion of a generalized subformula: $A$ is a \emph{$\mathcal{S}$-subformula of $B$} if 
$A\in\sub_{\mathcal{S}}(B)=\sub(B)\cup\{S(B'):S\in\mathcal{S},B'\in\sub(B)\}$. We say that \emph{$R$ is an $\mathcal{S}$-analytic calculus} if whenever $\Gamma\der_R\Delta$ then there exists a proof of $\Delta$ from $\Gamma$ using only formulas in $\sub_{\mathcal{S}}(\Gamma\cup\Delta)$. For finite $\mathcal{S}$, we have shown in~\cite{synt,wollic} that $\mathcal{S}$-analyticity implies that deciding $\der_R$ is in $\mathsf{coNP}$, and that proof-search can be implemented in $\mathsf{EXPTIME}$.\smallskip

Producing analytic calculi for logics characterized by finite PNmatrices is possible, as long as the syntax of the logic is sufficiently expressive (a notion intimately connected with the methods in~\cite{ShoesmithSmiley,AvronBK07,avroncutfree2013,dyadic,taming}). Fix a $\Sigma$-PNmatrix $\Mt=\tuple{V,D,\cdot_\Mt}$.
A pair of non-empty sets of elements $\emptyset\neq X,Y\subseteq V$ are \emph{separated}, $X\# Y$, if $X\subseteq D$ and $Y\subseteq V\setminus{D}$, or vice versa. A  formula $S$ with $\var(S)\subseteq\{p\}$ with $S_\Mt(z)\neq\emptyset$ for every $z\in V$, and such that $S_\Mt(x)\# S_\Mt(y)$ is said to \emph{separate} $x$ and $y$, and 
 called a \emph{(monadic) separator}. 
The PNmatrix $\Mt$ is said to be \emph{monadic} if there is a separator for every pair of distinct truth-values.\smallskip

Granted a monadic PNmatrix $\Mt=\tuple{V,D,\cdot_\Mt}$ and some set $\mathcal{S}=\{S^{xy}:x,y\in V, x\neq y\}$ of monadic separators for $\Mt$ such that each $S^{xy}$ separates $x$ and $y$, a
 \emph{discriminator} for $\Mt$ is the
$V$-indexed family
$\widetilde{\mathcal{S}}=\{\widetilde{\mathcal{S}}_x\}_{x\in V}$,
with each
$\widetilde{\mathcal{S}}_x=\{S^{xy}:y\in V\setminus\{x\}\}$. 
Each $\widetilde{\mathcal{S}}_x$ is naturally partitioned into 
  $
 \Omega_x=\{S\in \widetilde{\mathcal{S}}_x:S_\Mt(x)\subseteq D\} \text{ and } 
 \mho_x=\{S\in \widetilde{\mathcal{S}}_x:S_\Mt(x)\subseteq V\setminus D\}.
$ This partition is easily seen to characterize precisely each of the truth-values of $\Mt$.\smallskip

Given $X\subseteq V$, we denote by 
  $\Omega^*_X$ any of the possible sets built by choosing one element from 
 each $\Omega_x$ for $x\in X$, that is, $\Omega^*_X\subseteq\bigcup_{x\in X}\Omega_x$ is such that $\Omega^*_X\cap\Omega_x\neq \emptyset$ for each $x\in X$.
 Analogously, we let $\mho^*_X$ denote any of the possible sets built by choosing one element from 
 each 
 $\mho_x$ 
 for $x\in X$, that is, $\mho^*_X\subseteq\bigcup_{x\in X}\mho_x$ is such that $\mho^*_X\cap\mho_x\neq \emptyset$ for each $x\in X$. 
 The following result is taken from~\cite{wollic19}.

\begin{theorem}\label{analyticity}
Let $\Mt=\tuple{V,D,\cdot_\Mt}$ be a monadic PNmatrix  with discriminator $\widetilde{\mathcal{S}}$.
Then, $R_\Mt^{\widetilde{\mathcal{S}}}=R_{\exists}\cup R_{\mathsf{D}}\cup R_\Sigma \cup {R_{\mathcal{T}}}$ is an $\mathcal{S}$-analytic calculus for $\der_\Mt$, where:

\begin{itemize}
 \item $R_{\exists}$ contains, for each $X\subseteq V$ and each possible $\mho^*_{X}$ and $\Omega^*_{V\setminus X}$, the rule
$$\frac{\;\mho^*_{X}(p)\;}{\Omega^*_{V\setminus X}(p)}$$ 
 
   
 \item $R_{\mathsf{D}}$
    contains, for each $x\in V$, the rule $$\frac{\Omega_x(p)}{\;p,\mho_x(p)\;}\mbox{ if } x\in D \quad\mbox{ or }\quad 
  \frac{\;\Omega_x(p),p\;}{\mho_x(p)}\mbox{ if } x\notin D$$
 
 \item $R_\Sigma=\bigcup_{\conn\in\Sigma}R_\conn$ where, for $\conn\in \Sigma^{(k)}$, $R_\conn$ contains, for each
 $x_1,\ldots,x_k\in V$ and  $y\notin \conn_\Mt(x_1,\ldots,x_k)$, the rule  
$$ 
\frac{\;\bigcup\limits_{1\leq i\leq k}\Omega_{x_i}(p_i)\,,\, \Omega_y(\conn(p_1\dots,p_k))\;}{\bigcup\limits_{1\leq i\leq k}\mho_{x_i}(p_i)\,,\,\mho_y(\conn(p_1\dots,p_k))}$$

 \item 
$R_{\mathcal{T}}$ contains, for each $X\subseteq V$ with $X\notin \mathcal{T}_\Mt$, the rule

 $$\frac{\bigcup\limits_{x_i\in X}\Omega_{x_i}(p_i)}{\bigcup\limits_{x_i\in X}\mho_{x_i}(p_i)}.$$

 \end{itemize}

\end{theorem}

It is worth understanding the role of each of the rules proposed, as they fully capture the behaviour of $\Mt$. Namely, $R_\exists$ allows one to exclude combinations of separators that do not correspond to truth-values. Actually, in examples where the separators $S$ are such that, in all cases, $S_\Mt(z)\subseteq D$ or 
$S_\Mt(z)\subseteq V\setminus{D}$, one can always in practice set up the discriminator in a way that makes all $R_\exists$ rules trivial, in the sense that they will necessarily have a formula that appears both as a premise and as a conclusion. Rules in $R_{\mathsf{D}}$ distinguish those combinations of separators that characterize designated values from those that characterize undesignated values. Again, in practice, whenever $\Mt$ has both designated and undesignated values and $S(p)=p$ is used to separate them, all $R_{\mathsf{D}}$ rules are also trivial. The most operational rules are perhaps $R_\Sigma$, as they completely determine the interpretation of connectives in $\Mt$. The rules in $R=R_\exists\cup R_{\mathsf{D}}\cup R_\Sigma$ already guarantee that $\der_R{=}\der_\Mt$, but not necessarily analyticity. The rules in  $R_{\mathcal{T}}$ are crucial in proving analyticity (they are already derivable from the previous rules, but with seemingly non-analytic proofs). Indeed, rules in $R_{\mathcal{T}}$ guarantee that one deals with combinations of separators that correspond to values taken within a total refinement of $\Mt$. \\

In order to be able to apply this general result to obtain analytic calculi for the logics characterized by the PNmatrices produced by the method we have devised in Subsection~\ref{sec32}, we need to make sure that the PNmatrices are monadic. Of course, not every PNmatrix is monadic, but we can easily show that our construction preserves monadicity.

\begin{proposition}\label{monadic}
Let $\Mt=\tuple{V,D,\cdot_\Mt}$ be a $\Sigma$-PNmatrix and $\Ax\subseteq L_\Sigma(P)$ that fulfill the conditions of Theorem~\ref{theconstruction}.
If $\Mt$ is monadic then $\Mt^\sharp_\Ax$ is also monadic.
 \end{proposition}
\proof{Let $f_{v_1}^{A_1},f_{v_2}^{A_2}\in V^\sharp_\Ax$ with $f_{v_1}^{A_1}\neq f_{v_2}^{A_2}$. This means that there exists $w\in\Theta_\Ax$ such that $x_1=f_{v_1}^{A_1}(w)\neq f_{v_2}^{A_2}(w)=x_2$. Given that $\Mt$ is monadic, we know that there exists $S\in L_\Sigma(\{p\})$ which separates $x_1$ from $x_2$ in $\Mt$, that is, $S_\Mt(x_1)\#S_\Mt(x_2)$. We show that $R(p)=S(w\,p)$ separates $f_{v_1}^{A_1}$ from $f_{v_2}^{A_2}$ in $\Mt^\sharp_\Ax$.

Given $f_v^B\in V^\sharp_\Ax$ we know (from the completeness part of the proof of Theorem~\ref{theconstruction}) that $v^\sharp(C)=f_v^C$ for each $C\in L_\Sigma(P)$ defines a valuation $v^\sharp\in \Val_{V^\sharp_\Ax}$. Easily, then, $v^\sharp(R(B))\in R_{\Mt^\sharp_\Ax}(v^\sharp(B))=R_{\Mt^\sharp_\Ax}(f_v^B)$, and therefore $R_{\Mt^\sharp_\Ax}(f_v^B)\neq\emptyset$.

In order to show that $R_{\Mt^\sharp_\Ax}(f_{v_1}^{A_1})\# R_{\Mt^\sharp_\Ax}(f_{v_2}^{A_2})$ we just need to show that $R_{\Mt^\sharp_\Ax}(f_{v_1}^{A_1})(\varepsilon)\subseteq S_\Mt(x_1)\# S_\Mt(x_2)\supseteq R_{\Mt^\sharp_\Ax}(f_{v_2}^{A_2})(\varepsilon)$, and use the fact that in a rexpansion designated values are preserved and reflected. 

Take $i\in\{1,2\}$ and any valuation $v^\sharp\in \Val_{\Mt^\sharp_\Ax}$ with $v^\sharp(p)=f_{v_i}^{A_i}$. 
We have that $v^\sharp(R(p))=v^\sharp(S(w\,p))\in S_{\Mt^\sharp_\Ax}(v^\sharp(w\,p))$. Thus, it follows that 
$v^\sharp(R(p))(\varepsilon)\in S_{\Mt^\sharp_\Ax}(v^\sharp(w\,p))(\varepsilon)\subseteq S_\Mt(v^\sharp(w\,p)(\varepsilon))=S_\Mt(v^\sharp(p)(w))=S_\Mt(f_{v_i}^{A_i}(w))=S_\Mt(x_i)$.
\qed}

Note that this result encompasses the \emph{sufficient expressiveness} preservation result of~\cite{taming}, as the two-valued Boolean matrix is trivially separable using just $S(p)=p$.\smallskip

We now illustrate the powerful result of Theorem~\ref{analyticity} by producing suitably analytic calculi for the resulting logics in each of the examples of Section~\ref{sec4}. In some cases, we also take the opportunity to illustrate the (obvious) notion of proof in multiple-conclusion calculi.
In each of the examples, rules $R_\exists$ and $R_\mathsf{D}$ are omitted, as they are all trivial, as discussed before. We refer the reader to~\cite{synt,wollic19} for further details.

 \begin{bla1}\em
 
In Example~\ref{pnegpq} we have obtained a four-valued PNmatrix characterizing the strengthening of the logic of classical implication with the additional axiom $p_1\to(\neg p_1\to p_2)$. Easily, ${\mathcal S}=\{p,\neg p\}$ is a corresponding set of monadic separators, which yields the discriminator $\widetilde{\mathcal{S}}$ with $\widetilde{\mathcal{S}}_x={\mathcal S}$ for each truth-value $x$. This gives rise to the following partitions.

  \begin{center}
 \begin{tabular}{c|c|c}
 $x$ & $\Omega_x$ & $\mho_x$ \\
 \hline
 $00$ & $\emptyset$ & $\{p,\neg p\}$ \\
 $01$ & $\{\neg p\}$ & $\{ p\}$ \\
 $10$ & $\{p\}$ & $\{\neg p\}$  \\
 $11$  & $\{p,\neg p\}$ & $\emptyset$
\end{tabular}
\end{center}

Using Theorem~\ref{analyticity}, the following rules constitute an ${\mathcal S}$-analytic calculus $R$ for the logic.
  $$\frac{}{p\,,\,p\to q}\ _{r_1}\qquad\quad \frac{p\,,\,p\to q}{q}\ _{r_{2}}\qquad\quad \frac{q}{p\to q}\ _{r_{3}}
 \qquad\quad\frac{p \,,\, \neg p}{ q }\ _{r_{\Exp}}$$
 
After simplifications, the rules $r_{1}$--$r_{3}$ correspond to $R_\to$, and $r_{\Exp}$ to $R_{\mathcal{T}}$ with $X=\{00,11\}$, $X=\{01,11\}$, and $X=\{10,11\}$.

For illustration, we next depict an analytic proof of $\der_R p_1\to(\neg p_1\to p_2)$. Note that rules with multiple conclusions give rise to branching in the proof-tree, which makes it necessary for the target formula $p_1\to(\neg p_1\to p_2)$ to appear in all the branches.\\
\begin{center}

\scalebox{.85}{
\frame{\begin{tikzpicture}
\tikzset{level distance=49pt}
\tikzset{sibling distance=5pt}
 \Tree[.  $\emptyset$
  \edge node[auto=right] {$_{r_1}$};
  [.$p_1$ 
      \edge node[auto=right] {$_{r_1}$};
  [.$\neg p_1$ 
    \edge node[auto=right] {$_{r_{\Exp}}$}; 
   $p_1\to(\neg p_1\to p_2)$
  ]
   [.$\neg p_1\to p_2$
     \edge node[auto=right] {$_{r_3}$}; 
     $p_1\to(\neg p_1\to p_2)$
  ]
  ] 
  [.$p_1\to(\neg p_1\to p_2)$
    ]  
   ] 
\end{tikzpicture}}
}

 \end{center}
%
\hfill$\triangle$
\end{bla1}

\begin{bla2}\em
 In Example~\ref{pnegpnegq} we have obtained a four-valued PNmatrix characterizing the strengthening of the logic of classical implication with the additional axiom $p_1\to(\neg p_1\to \neg p_2)$. Easily, one can reuse the set of monadic separators, and the discriminator, from the previous example.
 
Using Theorem~\ref{analyticity}, an ${\mathcal S}$-analytic calculus $R$ for the logic can be obtained by replacing the rule $r_{\Exp}$ of Example~\ref{pnegpq} with the rule below.  $$\qquad\qquad\qquad\frac{p \,,\, \neg p}{ \neg q }\ _{r_{\Exp_\neg}}\qquad \qquad \qquad $$

Expectedly, rule $r_{\Exp_\neg}$ corresponds to $R_{\mathcal{T}}$ with $X=\{00,11\}$, and $X=\{10,11\}$.
\hfill$\triangle$
\end{bla2}

\begin{bla3}\em
  
 In Example~\ref{clun} we have obtained a four-valued Nmatrix characterizing $\mathcal{C}_{\min}$, the strengthening of the logic $\mathcal{CL}u\mathcal{N}$ with the additional axiom $\neg\neg p_1\to p_1$. Easily, ${\mathcal S}=\{p,\neg p,\neg\neg p\}$ is a corresponding set of monadic separators, which allows for the discriminator $\widetilde{\mathcal{S}}$ with $\widetilde{\mathcal{S}}_{010}=\{p\}$, $\widetilde{\mathcal{S}}_{101}=\{p,\neg p\}$, and $\widetilde{\mathcal{S}}_{110}=\widetilde{\mathcal{S}}_{111}=\{p,\neg p,\neg \neg p\}$, giving rise to the following partitions.

  \begin{center}
 \begin{tabular}{c|c|c}
 $x$ & $\Omega_x$ & $\mho_x$ \\
 \hline
 $010$ & $\emptyset$ & $\{p\}$ \\
 $101$ & $\{p\}$ & $\{\neg p\}$ \\
 $110$ & $\{p,\neg p\}$ & $\{\neg\neg p\}$  \\
 $111$  & $\{p,\neg p,\neg\neg p\}$ & $\emptyset$
\end{tabular}
\end{center}

Using Theorem~\ref{analyticity}, the following rules constitute an ${\mathcal S}$-analytic calculus $R$ for $\mathcal{C}_{\min}$.
 $$\frac{}{p\,,\,p\to q}\ _{r_1}\qquad \frac{p\,,\,p\to q}{q}\ _{r_{2}}\qquad \frac{q}{p\to q}\ _{r_{3}}\qquad\frac{ }{p \,,\, \neg p}\ _{r_4}\qquad\frac{ \neg \neg p}{p}\ _{r_5}$$
 
 After simplifications, the rules $r_{1}$--$r_{3}$ correspond to $R_\to$, and $r_4,r_5$ to $R_\neg$.\smallskip

We then obtained a three-valued Nmatrix characterizing the strengthening of $\mathcal{C}_{\min}$ with the axiom $ p_1\to\neg\neg p_1$.
Easily, ${\mathcal S}'=\{p,\neg p\}$ is a corresponding set of monadic separators, which allows for the discriminator $\widetilde{\mathcal{S}'}$ with $\widetilde{\mathcal{S}'}_{01}=\{p\}$, and $\widetilde{\mathcal{S}}_{10}=\widetilde{\mathcal{S}}_{11}=\{p,\neg p\}$, giving rise to the following partitions.

  \begin{center}
 \begin{tabular}{c|c|c}
 $x$ & $\Omega_x$ & $\mho_x$ \\
 \hline
 $01$ & $\emptyset$ & $\{p\}$ \\
 $10$ & $\{p\}$ & $\{\neg p\}$ \\
 $11$  & $\{p,\neg p\}$ & $\emptyset$
\end{tabular}
\end{center}

Using Theorem~\ref{analyticity}, an ${\mathcal S}'$-analytic calculus $R'$ for the logic can be obtained by joining to the calculus $R$ obtained above the new $ R_{\neg}$ rule:
 $$\frac{ p}{ \neg \neg p}$$
 \hfill$\triangle$
\end{bla3}

\begin{bla4}\em
In Example~\ref{agata} we have obtained a four-valued PNmatrix characterizing the strengthening of positive classical logic with axioms
$$p_1\ou \neg p_1$$ $$p_1\to(\neg p_1\to(\circ p_1 \to p_2))$$ $$\circ p_1\ou (p_1 \e \neg p_1)$$ $$\circ p_1\to\circ(p_1\e p_2)$$ $$(\neg p_1\ou \neg p_2)\to\neg(p_1\e p_2)$$

It is easy to see that ${\mathcal S}=\{p,\neg p,\circ p\}$ is a corresponding set of monadic separators, which allows for the discriminator $\widetilde{\mathcal{S}}$ with $\widetilde{\mathcal{S}}_{011}=\{p\}$,  $\widetilde{\mathcal{S}}_{101}=\{p,\neg p\}$, and $\widetilde{\mathcal{S}}_{110}=\widetilde{\mathcal{S}}_{111}=\{p,\neg p,\circ p\}$. This gives rise to the following partitions.
  
  \begin{center}
 \begin{tabular}{c|c|c}
 $x$ & $\Omega_x$ & $\mho_x$ \\
 \hline
 $011$ & $\emptyset$ & $\{p\}$ \\
 $101$ & $\{p\}$ & $\{\neg p\}$ \\
 $110$ & $\{p,\neg p\}$ & $\{\circ p\}$  \\
 $111$  & $\{p,\neg p,\circ p\}$ & $\emptyset$
\end{tabular}
\end{center}

Using Theorem~\ref{analyticity}, the following rules constitute an ${\mathcal S}$-analytic calculus $R$ for the logic.
    $$\frac{p\,,\, q}{\;p\e q\;}\ _{r_1}\quad 
 \frac{\;p\e q\;}{p}\ _{r_2}\quad \frac{\;p\e q\;}{q}\ _{r_3} 
 \quad \frac{\neg p}{\;\neg(p\e q)\;}\ _{r_4}  $$
   $$\frac{p}{\;p\ou q\;}\ _{r_5}\quad \frac{q}{\;p\ou q\;}\ _{r_6}\quad 
      \frac{\;p\ou q\;}{\;p\,,\, q\;}\ _{r_{7}}\quad\frac{\;p\,,\, p\to q\;}{q}\ _{r_{8}}\quad
    \frac{q}{\;p\to q\;}\ _{r_{9}}\quad 
   \frac{}{\;p\,,\, p\to q\;}\ _{r_{10}}
  $$ 
    $$\frac{}{\;p\,,\,\neg p\;}\ _{r_{11}}\quad\frac{}{\;p\,,\, \circ p\;}\ _{r_{12}}\quad \frac{p}{\;\neg p\,,\, \circ p\;}\ _{r_{13}} \quad\frac{\;p\,,\, q\,,\,\neg q\;}{\neg p}\ _{r_{14}}\quad \frac{\;p\,,\,\neg p\,,\, \circ p\;}{q}\ _{r_{15}}$$

After simplifications, the rules $r_{1}$--$r_{4}$ correspond to $R_\e$, $r_{5}$--$r_{7}$ to $R_\ou$, 
$r_{8}$--$r_{10}$ to $R_\to$,
$r_{11}$ to $R_\neg$, $r_{12}$ and $r_{13}$ to $R_\circ$. Finally, 
$r_{14}$ and $r_{15}$ result from $R_{\mathcal{T}}$, with $X=\{101,110\}$ and $X=\{111,011\}$, respectively.
%
%

%
 
Sample proofs, namely for some of the axioms, with a very similar calculus can be found in~\cite{wollic19}.
\hfill$\triangle$
\end{bla4}

\begin{bla5}\em
In Example~\ref{nelson} we have obtained a four-valued twist-structure characterizing the addition of a paraconsistent Nelson-like strong negation to positive classical logic. Easily, ${\mathcal S}=\{p,{\sim{p}}\}$ is a corresponding set of monadic separators, yielding the discriminator $\widetilde{\mathcal{S}}$ with $\widetilde{\mathcal{S}}_x=\mathcal{S}$ for each truth-value $x$. This gives rise to the following partitions.

  \begin{center}
 \begin{tabular}{c|c|c}
 $x$ & $\Omega_x$ & $\mho_x$ \\
 \hline
 $00$ & $\emptyset$ & $\{p,\sim p\}$ \\
 $01$ & $\{\sim p\}$ & $\{ p\}$ \\
 $10$ & $\{p\}$ & $\{\sim p\}$  \\
 $11$  & $\{p,\sim p\}$ & $\emptyset$
\end{tabular}
\end{center}

Using Theorem~\ref{analyticity}, the following rules constitute an ${\mathcal S}$-analytic calculus $R$ for the logic.
%
%
{
  $$
 \frac{\;p\e q\;}{p}\ _{r_1}\quad \frac{\;p\e q\;}{q}\ _{r_2}
 \quad \frac{p\,,\, q}{\;p\e q\;}\ _{r_3} \qquad \frac{ \sim p }{\sim (p\e q)}\ _{r_4} \qquad \frac{  \sim q }{\sim (p\e q)}\ _{r_5}\quad 
  \frac{\sim (p\e q)}{ \sim p\,,\, \sim q }\ _{r_6}
    $$
 $$ \frac{p}{\;p\ou q\;}\ _{r_7}\quad \frac{q}{\;p\ou q\;}\ _{r_8}\quad 
      \frac{\;p\ou q\;}{\;p\,,\, q\;}\ _{r_{9}}
       \quad  \frac{\sim (p\ou q)}{  \sim q }\ _{r_{10}}  \quad  \frac{\sim (p\ou q)}{  \sim q }\ _{r_{11}}\quad   
 \frac{ \sim p\,,\, \sim q }{\sim (p\ou q)}\ _{r_{12}}$$
   $$\quad\frac{\;p\,,\, p\to q\;}{q}\ _{r_{13}}\quad
    \frac{q}{\;p\to q\;}\ _{r_{14}}\quad 
   \frac{}{\;p\,,\, p\to q\;}\ _{r_{15}}$$
   $$
    \frac{\sim (p\to q)}{p }\ _{r_{16}}\quad
      \frac{\sim (p\to q)}{\sim q }\ _{r_{17}}\quad 
  \frac{p \,,\, \sim q }{\sim (p\to q)}\ _{r_{18}} 
  $$ 
 $$   \frac{ p }{\sim \sim p}\ _{r_{19}}\qquad \frac{ \sim \sim p }{ p}\ _{r_{20}} $$
}

After simplifications, the rules $r_{1}$--$r_{6}$ correspond to $R_\e$, $r_{7}$--$r_{12}$ to $R_\ou$, 
$r_{13}$--$r_{18}$ to $R_\to$,
$r_{19}$ and $r_{18}$ to $R_\sim$.

A strengthening with an additional (explosion) axiom $\sim p_1\to(p_1\to p_2)$ was then shown to be characterized by a four-valued Pmatrix. 
It is straightforward to see that one can reuse the set of monadic separators, and the discriminator, from above.
Using Theorem~\ref{analyticity}, an ${\mathcal S}$-analytic calculus $R'$ for the logic can be obtained by simply adding to $R$ the new $R_{\mathcal{T}}$ rule
$$\frac{p\,,\sim p}{q}$$
obtained by considering $X=\{11,00\}$,  $X=\{11,01\}$, and $X=\{11,10\}$.
\hfill$\triangle$
\end{bla5}

\begin{bla6}\em
In Example~\ref{swap} we obtained a four-valued Nmatrix characterizing the non-normal modal logic of Kearns and Ivlev~\cite{kearns,Ivlev}. It is not difficult to check (namely, using Proposition~\ref{monadic}) that ${\mathcal S}=\{p,\square p, \square\neg p\}$ is a set of monadic separators
for the Nmatrix. This allows for the discriminator $\widetilde{\mathcal{S}}$ with $\widetilde{\mathcal{S}}_{000}=\widetilde{\mathcal{S}}_{001}=\{p,\square\neg p\}$, and $\widetilde{\mathcal{S}}_{100}=\widetilde{\mathcal{S}}_{110}=\{p,\square p\}$, which gives rise to the following partitions.
  \begin{center}
 \begin{tabular}{c|c|c}
 $x$ & $\Omega_x$ & $\mho_x$ \\
 \hline
 $000$ & $\emptyset$ & $\{p,\square\neg p\}$ \\
 $001$ & $\{\square\neg p\}$ & $\{ p\}$ \\
 $100$ & $\{p\}$ & $\{\square p\}$  \\
 $110$  & $\{p,\square p\}$ & $\emptyset$
\end{tabular}
\end{center}
Using Theorem~\ref{analyticity}, we get an ${\mathcal S}$-analytic calculus $R$ for the logic.
{
 $$\frac{}{p\,,\,p\to q}\ _{r_1}\qquad\quad \frac{p\,,\,p\to q}{q}\ _{r_{2}}\qquad \frac{q}{p\to q}\ _{r_{3}}\qquad \frac{p\,,\, \neg p}{}\ _{r_4} \qquad  \frac{}{p\,,\, \neg p}\ _{r_5} 
 $$
  $$\frac{\square(p\to q) \,,\, \square p}{\square q}\ _{k}
  \qquad
   \frac{\square(p\to q) \,,\, \square \neg q}{\square\neg p}\ _{k_1}
   \qquad
   \frac{\square p \,,\, \square\neg q}{\square\neg(p\to q)}\ _{k_2}
    $$
  $$
   \frac{\square\neg p}{\square(p\to q)} \ _{m_1}\qquad
    \frac{\square q}{\square(p\to q)} \ _{m_2}\qquad
     \frac{\square\neg (p\to q)}{\square \neg q} \ _{m_3}\qquad
     \frac{\square\neg (p\to q)}{\square p} \ _{m_4}   
    $$
  $$
   \frac{\square p}{p} \ _{T}\qquad
    \frac{\square p}{\square \neg \neg p} \ _{dn_1} \qquad
      \frac{\square \neg \neg p}{\square p} \ _{dn_2} 
    $$    
    
}

After simplifications, the rules $r_{1}$--$r_{3}$, $k$, $k_1$--$k_2$, $m_1$--$m_4$ correspond to $R_\to$, $r_{4}$--$r_{5}$ and $dn_1$--$dn_2$ to $R_\neg$, 
and $T$ to $R_\square$. It is interesting to note that rules $r_{1}$--$r_{5}$ characterize classical logic, and the remaining rules are in a one-to-one correspondence with the axioms considered (see~\cite{coniglioswap}). The only less obvious case is the rule $k_2$. For this reason we present below an analytic proof of the corresponding axiom $K_2=\neg\square\neg(p\to q)\to(\square p\to\neg\square\neg q)$, i.e.,  $\der_R  K_2$. Note that $K_2$ is obtained in all the branches of the proof-tree, except for the leftmost one, which is discontinued due to rule $r_4$ (as signaled by the use of $\ast$).

\begin{center}
%
\scalebox{.85}{
\frame{\begin{tikzpicture}
\tikzset{level distance=49pt}
\tikzset{sibling distance=5pt}
 \Tree[.  $\emptyset$
  \edge node[auto=right] {$_{r_1}$};
  [.$\neg\square\neg (p\to q)$ 
      \edge node[auto=right] {$_{r_1}$};
  [.$\square p$ 
    \edge node[auto=right] {$_{r_5}$};
  [.$\square \neg q$ 
      \edge node[auto=right] {$_{k_2}$};
  [.$\square\neg (p\to q)$ 
    \edge node[auto=right] {$_{r_4}$}; 
    $\ast$
  ] 
  ] 
  [.$\neg \square \neg q$
   \edge node[auto=right] {$_{r_3}$}; 
   [.$\square p\to\neg \square\neg q$
   \edge node[auto=right] {$_{r_3}$}; 
  $K_2$
    ]  
    ]  
  ]
   [.$\square p\to\neg \square\neg q$
     \edge node[auto=right] {$_{r_3}$}; 
     $K_2$
  ]
  ] 
  [.$K_2$
    ]  
   ] 
\end{tikzpicture}}
}
 \end{center}
\hfill$\triangle$
\end{bla6}

\begin{bla7}\em

%
%
%
%

In Example~\ref{L53} we have obtained the usual three-valued {\L}ukasiewicz's matrix (by strengthening the five-valued  {\L}ukasiewicz logic with an additional axiom). Easily, ${\mathcal S}=\{p,\neg p\}$ is a set of monadic separators, yielding the  discriminator $\widetilde{\mathcal{S}}$ with 
$\widetilde{\mathcal{S}}_{0}=\widetilde{\mathcal{S}}_{\frac{1}{2}}=\{p,\neg p\}$, and $\widetilde{\mathcal{S}}_{1}=\{p\}$, which gives rise to the following partitions.

  \begin{center}
 \begin{tabular}{c|c|c}
 $x$ & $\Omega_x$ & $\mho_x$ \\
 \hline
 $0$ & $\{\neg p\}$ & $\{p\}$ \\
 $\frac{1}{2}$ & $\emptyset$ & $\{p,\neg p\}$  \\
 $1$  & $\{p\}$ & $\emptyset$
\end{tabular}
\end{center}

%

Using Theorem~\ref{analyticity}, the following rules constitute an ${\mathcal S}$-analytic calculus $R$ for $\mathcal{L}_3$.

 $$ \frac{p\,,\, \neg p}{}\ _{r_1} \qquad\frac{p}{ \neg \neg  p}\ _{r_2} \qquad \quad \frac{ \neg \neg p}{p}\ _{r_3} $$

$$\frac{}{p\,,\,p\to q\,,\,\neg q}\ _{r_4}\qquad\quad \frac{p\,,\,p\to q}{q}\ _{r_{5}}\qquad \frac{q}{p\to q}\ _{r_{6}}$$
  
 $$ \frac{ \neg p}{p\to q}\ _{r_7}   \qquad \quad  \frac{\neg q \,,\, p\to q   }{\neg p }\ _{r_8} $$

$$ \frac{ \neg (p\to q)}{p}\ _{r_9}   \qquad \quad  \frac{ \neg (p\to q)}{\neg q}\ _{r_{10}}  \qquad \quad  \frac{p\,,\,\neg q}{ \neg (p\to q)}\ _{r_{11}} $$

After simplifications, the rules $r_{1}$--$r_{3}$ correspond to $R_\neg$, and $r_{4}$--$r_{11}$ to $R_\to$. 
For illustration, we depict an analytic proof of the added axiom $A=((p \to \neg p) \to p) \to p)$, i.e., $\der_R A$.\\

%
%
%
 
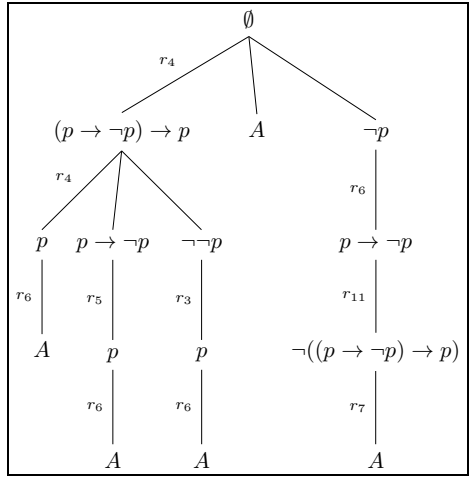
\begin{center}

\scalebox{.85}{
\frame{\begin{tikzpicture}
\tikzset{level distance=49pt}
\tikzset{sibling distance=5pt}
 \Tree[. $\emptyset$
  \edge node[auto=right] {$_{r_4}$};
  [.$(p\to \neg p)\to p$ 
  \edge node[auto=right] {$_{r_4}$};
   [.$p$ 
   \edge node[auto=right] {$_{r_6}$};
   $A$
   ]
   [.$p\to \neg p$ 
   \edge node[auto=right] {$_{r_5}$};
   [.$p$ 
   \edge node[auto=right] {$_{r_6}$};
   $A$
   ]
]
   [.$\neg \neg p$ 
   \edge node[auto=right] {$_{r_3}$};
   [.$p$ 
   \edge node[auto=right] {$_{r_6}$};
   $A$
   ]
   ] 
  ] 
 [.$A$
   ]  
  [.$\neg p$
   \edge node[auto=right] {$_{r_{6}}$};
   [.$p\to \neg p$
   \edge node[auto=right] {$_{r_{11}}$};
   [.$\neg((p\to \neg p)\to p)$
    \edge node[auto=right] {$_{r_7}$};
   $A$
   ]
   ]
   ]  ] 
\end{tikzpicture}}
}
\end{center}

\hfill$\triangle$
\end{bla7}

\section{Concluding remarks}\label{sec6}

In this paper we have shown that rexpansions of (P)(N)matrices are a universal tool for explaining the strengthening of logics with additional axioms. 
This does not come as a surprise, as non-determinism and partiality are well known for enabling a plethora of compositionality results in logic. Our general method in Theorem~\ref{flat} is not effective, but it still brings about some interesting phenomena, such as the possibility of building a denumerable semantics for intuitionistic propositional logic (where the precise roles of non-determinism and partiality need further clarification). More practical, though, is our less general method in Theorem~\ref{theconstruction} as, despite the necessary restrictions on its scope, it brings about an effective method for producing finite semantic characterizations whenever starting from a finite basis. Our results cover a myriad of examples in the literature, namely those motivated by the study of logics of formal inconsistency, which played an important role in the work of Arnon Avron. Besides, our effective method, while more general and incremental, is fully inspired by the fundamental ideas in~\cite{taming}. It is also worth noting that our results apply not just to the Tarskian notion of consequence relation, but also to the multiple-conclusion case. An obvious topic for further work is to provide a usable tool implementing these methods.

Other opportunities for further research, aimed at generalizing the results presented, would be to find more general syntactic conditions on the set of allowed axioms. For instance, the number of sentential variables occurring in an axiom seems to be easy to flexibilize by artificially extending the logic with big-arity connectives. Beyond axioms, one could think even further away, and consider strengthening logics with fully-fledged inference rules. In any case, such extensions will expectedly need more sophisticated techniques than the simple idea behind \emph{look-aheads}.

These results reinforce the need to better understand the conditions under which two (P)(N)matrices characterize the same logic. This is by no means a trivial question, but we believe that the notion of rexpansion can be a useful tool in that direction.\smallskip

If not for its own sake, this line of research aimed at providing effective semantic characterizations for combined logics is quite well justified by 
another recurring goal of many of the papers that inspired us: ultimately obtaining suitably analytic calculi for the resulting logics.

\bibliographystyle{plain}  
\bibliography{biblio.bib}

\end{document}